\documentclass[a4paper,10pt]{article}
\usepackage{amsfonts}
\usepackage{amsmath}
\usepackage{amsthm}
\usepackage{amssymb}
\newcommand{\br}{\mathbb R}

\def\eqldef{\overset{\text{\tiny def}}{=}}

\newtheorem{lemma}{Lemma}
\newtheorem{theorem}{Theorem}
\newtheorem{proposition}{Proposition}
\usepackage{pstcol}
\usepackage{pstricks}

\title{
Global existence for a 3D non-stationary Stokes flow with Coulomb's type friction boundary conditions}

\author{ 
 {Mahdi Boukrouche  
  and  Laetitia Paoli}\thanks{Lyon University, F-42023 
Saint-Etienne,
Institut Camille Jordan CNRS UMR 5208,
23 rue  Paul Michelon 42023 Saint-Etienne Cedex 2, France. 
Corresponding author: laetitia.paoli@univ-st-etienne.fr}
 }

\date{}

\begin{document}
\maketitle \normalsize

\begin{abstract}
In this paper we study non stationary viscous incompressible fluid flows with nonlinear boundary slip conditions given by a subdifferential property of friction type. More precisely we assume that the tangential velocity vanishes as long as the shear stress remains below a threshold ${\cal F}$, that may depend on the time and the position variables but also on the stress tensor, allowing to consider  Coulomb's type friction laws. An existence and uniqueness theorem is obtained first 
when the threshold ${\cal F}$ is a data and sharp estimates are derived for the velocity and pressure fields as well as for the stress tensor. Then an existence result is proved for the non-local Coulomb's friction case by using a successive approximation technique with respect to the shear stress threshold.
\end{abstract}


\bigskip
\noindent{\bf AMS :}
76D03,76D07, 35K86, 49J40, 35A35

 \renewcommand{\theequation}{1.\arabic{equation}}
 \setcounter{equation}{0}
 
 
  \section{Introduction} \label{intro}

  In the study of fluid flow, it is usually assumed that the fluid sticks to the boundary of the flow domain, leading to the so-called no-slip boundary condition. Such a behaviour has been mathematically justified by considering the microscopic asperities of the boundary (see \cite{richardson, bucur1, bucur2, bucur3, brezina}). Unfortunately experiments show that more complex boundary conditions may occur, especially in the case of non-wetting, hydrophobic or chemically patterned surfaces (\cite{barrat, baudry, tretheway, zhu2, hervet}), leading to linear slip conditions of Navier type (\cite{navier, bonaccurso, zhu1}) or nonlinear slip conditions of friction type when the tangential fluid velocity does not vanish only when a threshold is reached (\cite{magnin}). This last class of boundary conditions  has been considered first for Bingham fluids in \cite{ionescu-sofonea}. Then they have been introduced 
   for incompressible Newtonian  fluid flows by H.Fujita during his lectures at Coll\`ege de France in 1993 (\cite{fujita1}) and subsequently studied by H.Fujita who proved existence and uniqueness for the stationary Stokes problem and by N.Saito  who established some regularity properties for the solutions (\cite{fujita2, fujita4, saito2}). See also \cite{haslinger} for shape optimization issues.
  
   More precisely they consider  boundary conditions given by a subdifferential property, i.e.
\begin{eqnarray*}
\begin{array}{l}
v=0 \quad \hbox{ on} \  \Gamma \setminus \Gamma_0 \ \hbox{(no-slip condition on $\Gamma \setminus \Gamma_0$)}, \\
 v_n = 0   , \quad - \sigma_{t} \in {\cal F} \partial \bigl( |v_t| \bigr) \quad \hbox{ on} \quad \Gamma_0 \ \hbox{(slip condition on $\Gamma_0$)}
 \end{array}
\end{eqnarray*}
where the boundary of the flow domain is splitted into $ \Gamma = \Gamma_0 \cup (\Gamma \setminus \Gamma_0)$, 
${\cal F}$ is a given positive function on $\Gamma_0$, $\partial \bigl(| \cdot | \bigr)$ is the subdifferential of the function $| \cdot |$,
$v_n$, $v_t$ and $\sigma_t$ are the normal component of the velocity,  the tangential component of the velocity and the shear  stress respectively. By using the definition of the subdifferential of a convex function (\cite{rockafellar}), we may rewrite the boundary condition on $\Gamma_0$ as 
\begin{eqnarray*}
 v_n = 0   , \quad |\sigma_t| \le {\cal F}  \quad \hbox{ on} \quad \Gamma_0
\end{eqnarray*}
and  
\begin{eqnarray*} \label{intro.7}
\begin{array}{l}
|\sigma_t| < {\cal F} \Longrightarrow v_t = 0, \\
|\sigma_t| =  {\cal F} \Longrightarrow 
\exists\lambda\geq 0
\ \mbox{\rm s.t.}\  v_{t}=  -\lambda \sigma_{t}
\end{array}
\quad\mbox{\rm on}\quad \Gamma_0 
\end{eqnarray*}
which can be interpreted as a Tresca's friction condition on $\Gamma_0$ (\cite{duvaut2}). 

For the unsteady Stokes problem, existence  has been established by H.Fujita (\cite{fujita3}) when the density of body forces is equal to zero by using the non linear semigroup theory. For regularity properties the reader is referred to  N.Saito and  H.Fujita (\cite{saito1}).

\bigskip

The purpose of this paper is to extend these results to unsteady problems with  non-vanishing external forces and to  more general friction boundary conditions, like Coulomb's friction boundary conditions, where the  threshold ${\cal F}$ may depend on stress tensor $\sigma$. Indeed,  for solids in contact  with a sliding planar  surface $\Gamma_0$, Coulomb established experimentally (\cite{coulomb}) that
\begin{eqnarray*}
 v_n = 0   , \quad |\sigma_t| \le k |\sigma_n|  \quad \hbox{ on} \quad \Gamma_0
\end{eqnarray*}
and  
\begin{eqnarray*} \label{intro.7}
\begin{array}{l}
|\sigma_t| < k |\sigma_n| \Longrightarrow v_t = s, \\
|\sigma_t| =  k |\sigma_n| \Longrightarrow 
\exists\lambda\geq 0
\ \mbox{\rm s.t.}\  v_{t}=  s -\lambda \sigma_{t}
\end{array}
\quad\mbox{\rm on}\quad \Gamma_0 
\end{eqnarray*}
where $\sigma_n$ is the normal component of the stress vector, $s$ is the sliding velocity of the surface and $k>0$ is a friction coefficient.  Hence we get
\begin{eqnarray*}
 v_n = 0   , \quad - \sigma_{t} \in {\cal F} \partial \bigl( |v_t - s| \bigr) \quad \hbox{ on} \quad \Gamma_0
\end{eqnarray*}
with ${\cal F} =  k |\sigma_n|$.

\bigskip

In order to be able to take into account also possible anisotropic friction, we will consider in this paper shear stress thresholds of the form ${\cal F} ={\cal F} (x', t, \sigma)$.
 More precisely we consider a non-stationary Stokes flow described by the system
 \begin{eqnarray}\label{intro.1}
\frac{\partial v}{\partial t} -  div \bigl( 2 \mu D(v) \bigr) + \nabla p =  
f\quad\mbox{\rm in}\quad\Omega\times(0,T),
\end{eqnarray}
\begin{eqnarray}\label{intro.2}
div(v) = 0 \quad\mbox{\rm in}\quad\Omega\times(0,T),
\end{eqnarray}
 with the initial condition
\begin{eqnarray}\label{intro.3}
v(0) = v^{0}\quad\mbox{\rm in}\quad  \Omega.
\end{eqnarray}
Here $[0,T]$ is a given non-trivial time interval, $v$ and $p$ denote respectively the velocity and the pressure of the fluid, $\mu \in \br^*_+$ is its viscosity, $f$ is the density of body forces and $D(v)$ is the strain rate tensor defined as
\begin{eqnarray*}
D(v) = \bigl(d_{ij}(v) \bigr)_{1 \le i,j \le 3}, \quad 
d_{ij} (v) = \frac{1}{2} \left( \frac{\partial v_i}{\partial x_j} +  \frac{\partial v_j}{\partial x_i} \right) \quad 1 \le i,j \le 3.
\end{eqnarray*}
Motivated by lubrication or extrusion/injection problems, we assume that the fluid domain $\Omega$ is given by 
\begin{eqnarray*}
\Omega = \bigl\{(x',x_{3})\in \mathbb{R}^{2}\times\mathbb{R}
:\quad x'\in\omega,\quad 0< x_{3}< h(x')\bigr\},
\end{eqnarray*}
where $\omega$ is a non empty open bounded subset of $\br^{2}$ with a Lipschitz continuous boundary, 
and  $h$ is a  Lipschitz continuous function 
which is bounded from above and from below by some positive real numbers. 
We decompose the boundary of $\Omega$ as $\partial\Omega=\Gamma_0 \cup \Gamma_{L}\cup \Gamma_{1}$,
with
$\Gamma_0=\{(x',x_{3})\in\overline{\Omega}: x_{3}=0\}$, 
$\Gamma_{1}=\{(x',x_{3})\in\overline{\Omega}:
 x_{3}=h(x')\}$ and $\Gamma_{L}$  the lateral part of the boundary. Let us denote by $n=(n_{1}, n_{2}, n_{3})$ the unit outward normal  vector to  
$\partial \Omega$,
and by $u \cdot w$ (resp. $|u|$) the Euclidean inner product (resp. the Euclidean norm) of  vectors $u$ and $w$.
We define  the normal and the tangential velocities on 
$\partial\Omega$ by
\begin{eqnarray*}
v_{n}=v\cdot n = \sum_{i=1}^3 v_{i} n_{i},\quad v_{{t}} = \bigl( v_{{t}i} \bigr)_{1 
\le i \le 3}
\,  \mbox{\rm with} \,  \ v_{{t}i}=v_i -v_{n} n_i  \quad  1 \le i \le 3
\end{eqnarray*}
and the normal and the tangential components of the stress tensor $\sigma = -p {\rm Id} + 2 \mu D(v)$  by
\begin{eqnarray*}
\sigma_{n}  =\sum_{i,j =1}^3 \sigma_{ij} n_{i} n_{j},\quad 
\sigma_{{t}} = \bigl( \sigma_{{t} i}\bigr)_{1 \le i \le 3}
\  \mbox{with} \    \ \sigma_{{t} i}=\sum_{j=1}^3 \sigma_{ij} n_{j}-\sigma_{n} n_{i}  
\quad 1 \le i \le 3.
\end{eqnarray*}
 We introduce a  function $s: \Gamma_0 \to \br^2$ and a function  $g: \partial \Omega \to \mathbb{R}^3$ such that
\begin{eqnarray*}
\int_{\Gamma_{L}}g_n \,d \gamma =0, \quad g=0 \   \mbox{\rm  on }\  \Gamma_1 ,
\quad g_{n} = 0  \ \mbox{\rm  and } g_{t} = g- g_{n} n =(s, 0)  \  \mbox{\rm on } \Gamma_0.
\end{eqnarray*}

We assume that the upper part of the fluid domain is fixed while the lower part is moving with a shear velocity given by $s \zeta (t)$, where $\zeta: [0,T] \to \br$ is such that $\zeta (0) =1$. Then the fluid  velocity satisfies  the following non-homogeneous boundary condition on $\Gamma_1 \cup \Gamma_L$
\begin{eqnarray}\label{intro.4}
v = g \zeta\quad\mbox{on}\quad (\Gamma_1 \cup \Gamma_L) \times(0, T).
\end{eqnarray}
We assume furthermore that the flow satisfies a generalized dry friction law on $\Gamma_0$, i.e. 
\begin{eqnarray}\label{intro.5}
v_{n} =0\quad\mbox{\rm on}\quad \Gamma_0 \times (0,T),
\end{eqnarray}
\begin{eqnarray}\label{intro.6}
|\sigma_t|\le {\cal F} (x',t, \sigma ) \quad \mbox{\rm on}\quad \Gamma_0 \times (0,T)
\end{eqnarray}
and
\begin{eqnarray} \label{intro.7}
\begin{array}{l}
|\sigma_t| < {\cal F} (x',t, \sigma ) \Longrightarrow v_t = (s \zeta, 0), \\
|\sigma_t| =  {\cal F} (x',t, \sigma ) \Longrightarrow 
\exists\lambda\geq 0
\ \mbox{\rm s.t.}\  v_{t}= (s\zeta, 0) -\lambda \sigma_{t}
\end{array}
\quad\mbox{\rm on}\quad \Gamma_0 \times (0,T)
\end{eqnarray}
where ${\cal F}$ is a given non negative mapping on $\Gamma_0 \times (0,T) \times \br^{3 \times 3}$. 



\bigskip

\renewcommand{\theequation}{2.\arabic{equation}}
\setcounter{equation}{0}

\section{Mathematical formulation of the problem} \label{math_formulation}

In order to get a variational formulation of the problem we introduce the following functional spaces
\begin{eqnarray*}
{\cal V}_{0} = \left\{\varphi\in \bigl({ H}^{1}(\Omega)\bigr)^3:\  \varphi=0\ \mbox{on}\
\Gamma_{1}\cup\Gamma_{L}, \ \varphi_n=0\ \mbox{on}\ \Gamma_0\right\},
\end{eqnarray*}
\begin{eqnarray*}
{\cal V}_{0 div} = \bigl\{\varphi\in {\cal V}_0:\ div (\varphi) =0\ \mbox{in}\ 
\Omega\bigr\},
\end{eqnarray*}
 endowed  with the norm of 
  $\bigl( H^1(\Omega)\bigr)^3 $, 
and 
\begin{eqnarray*}
L^{2}_{0}(\Omega) = \left\{q\in  L^{2}(\Omega):\ \int_{\Omega} q\,dx = 
0\right\}
\end{eqnarray*}
endowed with the norm of $L^2(\Omega)$. 
We assume that 
\begin{eqnarray}\label{mu_et_f}
  f\in L^{2}\bigl(0,T; \bigl(L^{2}(\Omega))^3 \bigr),
\end{eqnarray}
\begin{eqnarray}\label{zeta}
\mu \in \br^*_+, \quad \zeta \in {\cal C}^{\infty} \bigl( [0, T], \br \bigr) 
   \mbox{ such that } \zeta (0)=1,
\end{eqnarray}
with $T>0$, and we define
\begin{eqnarray*}
\displaystyle a (u,v)  & = & \int_{\Omega} 2 \mu  D(u) : D(v) \, dx \\
&  = &
\displaystyle \int_{\Omega} 2 \mu \sum_{i,j =1}^3 d_{ij}(u)d_{ij}(v)\,dx \quad \forall (u, v)\in \bigl( {
H}^{1}(\Omega) \bigr)^3 \times \bigl({ H}^{1}(\Omega) \bigr)^3. 
\end{eqnarray*}
From Korn's inequality \cite{Korn} we infer that there exists $\alpha>0$ such that
\begin{eqnarray}\label{coercif}
\alpha \|u\|^2_{ {\bf H}^1 (\Omega)} \le \int_{\Omega} 2 \mu  D(u):D(u) \, dx  
\le  2 \mu \|u\|^2_{ {\bf H}^1 (\Omega)}
\quad \forall u \in  {\cal V}_{0}.
\end{eqnarray}

Moreover, in order to deal with homogeneous boundary conditions on $\Gamma_1 \cup \Gamma_L$, we 
assume  that  there exists an extension of $g$ to $\Omega$, denoted by $G_0$, 
such that
\begin{eqnarray}\label{G_0}
G_{0}\in \bigl( { H}^{2}(\Omega)\bigr)^3, \quad   div  (G_{0})=0 \  \mbox{in}\ \Omega, 
\quad  G_{0}=g\ \mbox{on}\ \partial \Omega,
\end{eqnarray}
and  we let
\begin{eqnarray*}\label{tildev}
 \widetilde{v}=v-G_{0}\zeta.
\end{eqnarray*}
By multiplying (\ref{intro.1}) by a test-function $\varphi \chi$, with $\varphi \in {\cal V}_0$ and $\chi \in {\cal D}(0,T)$, a formal integration by part  
leads to
\begin{eqnarray*}
\begin{array}{l}
\displaystyle - \int_0^T \int_{\Omega} div \bigl( 2 \mu D(v) \bigr) \cdot \varphi \chi \, dx dt + \int_0^T \int_{\Omega}  \nabla p \cdot \varphi \chi \, dx dt  \\
\displaystyle = - \int_0^T \int_{\Omega} div (\sigma) \cdot \varphi \chi  \, dx dt  \\
\displaystyle = \int_0^T \int_{\Omega} 2 \mu \sum_{i,j=1}^3 d_{ij} (v) \frac{\partial \varphi_i}{\partial x_j} \, dx dt - \int_0^T \int_{\Omega} p div(\varphi) \chi \, dx dt  \\
\displaystyle - \int_0^T \int_{\partial \Omega}\sum_{i, j = 1}^3 \sigma_{ij} \varphi_i n_j \chi \, d \gamma dt.
\end{array}
\end{eqnarray*}
Then we infer from (\ref{intro.5})-(\ref{intro.7}) that
\begin{eqnarray*}
\sigma_t \cdot( v - G_0 \zeta) + {\cal F} \bigl(x', t, \sigma \bigr) | v - G_0 \zeta | = 0 \quad \mbox{ on $\Gamma_0 \times (0,T)$}
\end{eqnarray*}
and recalling that $\varphi \in {\cal V}_0$  we get
\begin{eqnarray*}
\begin{array}{l}
\displaystyle \int_0^T \int_{\partial \Omega}\sum_{i, j = 1}^3 \sigma_{ij} \varphi_i n_j \chi \, d \gamma dt 
= \int_0^T \int_{\Gamma_0}  \sigma_{t} \cdot \varphi  \chi \, dx' dt \\
\displaystyle \ge - \int_0^T \int_{\Gamma_0}  {\cal F} \bigl(x', t, \sigma \bigr) |\widetilde v + \varphi \chi | \, dx' dt
+  \int_0^T \int_{\Gamma_0}  {\cal F} \bigl(x', t, \sigma_n \bigr) |\widetilde v  | \, dx' dt.
\end{array}
\end{eqnarray*}
Since we expect $\widetilde v$ to take its values in $L^2 \bigl( 0, T; {\cal V}_{0 div} \bigr)$ these integrals make sense if ${\cal F} (x', t, \sigma) \in L^{2} (0, T; L^{2}(\Gamma_0) \bigr)$.

Under this assumption  we consider the following problem

\noindent {\bf Problem $(P)$}
Find
$$\widetilde{v}\in L^{2} \bigl(0,T;{\cal V}_{0div} \bigr)\cap L^{\infty}\bigl(0,T; \bigl( L^{2}(\Omega) \bigr)^3 \bigr),\  \,
p\in H^{-1}\bigl(0,T ;L^{2}_{0}(\Omega) \bigr)$$
such that, for all $ \varphi\in {\cal V}_{0}$ and for all $\chi\in {\cal D}(0,T)$, we have
\begin{equation}\label{NS-25}
\begin{array}{ll}
\displaystyle
\left\langle \frac{d}{dt} \left( \widetilde{v}, \varphi \right) , \chi\right\rangle_{{\cal D}'(0,T), {\cal D}(0,T)}  
 - \bigl\langle \bigl(p,div(\varphi)  \bigr), \chi \bigr\rangle_{{\cal D}'(0,T), {\cal D}(0,T)} \\
\displaystyle + \int_0^T a(\widetilde{v},\varphi \chi ) \, dt +\Psi_{{\cal F}} (\widetilde{v}+ \varphi \chi )-\Psi_{{\cal F}} (\widetilde{v})
\displaystyle  \geq
\bigl\langle (f,\varphi ), \chi \bigr\rangle_{{\cal D}'(0,T), {\cal D}(0,T)}  \\
\displaystyle -  \int_0^T a(G_{0}\zeta ,\varphi \chi) \, dt  - \left\langle \left(G_{0}\frac{\partial \zeta}{\partial t},\varphi \right),
\chi \right\rangle_{{\cal D}'(0,T ), {\cal D}(0,T)} 
\end{array}
\end{equation}
where $\Psi_{{\cal F}}$ is given by
\begin{eqnarray*}
\Psi_{{\cal F}} (u) =
 \int_0^T \int_{\Gamma_0} {\cal F} ( x', t, \sigma) \bigl| u(x', t) \bigr| \, dx' dt
 \quad \forall u \in L^2 \bigl( 0,T; \bigl( L^2 (\Gamma_0) \bigr)^3 \bigr)
\end{eqnarray*}
together with the initial condition
\begin{eqnarray}\label{NS-25init}
\widetilde{v}(0, \cdot) = \widetilde{v}^{0} = v^0 - G_0 \zeta (0) = v^0 - G_0,
\end{eqnarray}
where
$(\cdot, \cdot)$ denotes the inner product in $\bigl( { L}^2(\Omega) \bigr)^3$. 
Let us emphasize that we identify $\widetilde{v} + \varphi \chi $ and $\widetilde{v} $ with their trace on $\Gamma_0$ in the definition of $\Psi_{{\cal F}}(\widetilde{v}+ \varphi \chi )$ and $\Psi_{{\cal F}}(\widetilde{v})$.

\bigskip

We may observe that, for any solution of problem $(P)$, the stress tensor $\sigma= -p {\rm Id} + 2 \mu D( {\widetilde v} + G_0 \zeta)$  belongs to $H^{-1} \bigl(0,T; \bigl( L^2(\Omega) \bigr)^{3 \times 3} \bigr)$. Thus we can not consider directly the Coulomb's friction case described by ${\cal F}(\cdot, \cdot, \sigma) = k |\sigma_n|$ since $\sigma_n$ is not necessarily well defined on $\Gamma_0$ and $|\sigma_n|$ does not belong to $L^2 \bigl( 0,T; L^2(\Gamma_0) \bigr)$. This kind of difficulty appears also in the study of frictional contact problems in solid mechanics and it has been encompassed by replacing $\sigma_n$ by some regularization $\sigma_n^*$. This idea introduced by G.Duvaut (\cite{duvaut2, duvaut3}) has led to the so-called non-local Coulomb's friction law. The regularization procedure $\sigma_n \mapsto \sigma_n^*$ is built by using a linear continuous operator from $H^{-1/2} ( \partial \Omega)$ to $L^2(\Gamma_0)$ which fits the mechanical meaning of $\sigma_n$ that is defined as the ratio of a force by a surface. Namely, for the static case in solid mechanics, it is easily proved by duality arguments that $\sigma_n$ belongs to  $H^{-1/2} ( \partial \Omega)$ and $\sigma_n^*$ is obtained by convolution of $\sigma_n$ with a smooth non-negative function (\cite{demkowicz}). See also \cite{consiglieri} in the framework of non-Newtonian fluids.

In our case, for the unsteady flow problem, we have to deal with two additional difficulties. Indeed, in order to define the normal trace of $\sigma$ on $\partial \Omega$, we have to establish some regularity properties for $div(\sigma)$. But, if we choose $\varphi \in \bigl({\cal D}(\Omega) \bigr)^3$ and $\chi \in {\cal D} (0,T)$ in (\ref{NS-25}) we get 
\begin{eqnarray*}
&&\left\langle \frac{d}{ d t} \left( \widetilde{v}, \varphi \right) ,\pm \chi\right\rangle_{{\cal D}'(0,T), {\cal D}(0,T)} 
+ \int_0^T a(\widetilde{v}+ G_0 \zeta,\varphi \chi) \, dt  \nonumber\\
&&
- \bigl\langle \bigl(p , div(\varphi)  \bigr), \pm \chi \bigr\rangle_{{\cal D}'(0,T), {\cal D}(0,T)} 
 \nonumber\\
&&
  \geq
\bigl\langle (f,\varphi ), \pm \chi \bigr\rangle_{{\cal D}'(0,T), {\cal D}(0,T)} 
 - \left\langle \left(G_{0} \frac{\partial \zeta}{\partial t},\varphi \right),
 \pm \chi \right\rangle_{{\cal D}'(0,T), {\cal D}(0,T)}  
\end{eqnarray*}
i.e.
\begin{eqnarray*}\label{sigma}
\begin{array}{ll}
\displaystyle 
\left\langle {d \over d t} \left( \widetilde{v}, \varphi \right) , \chi\right\rangle_{{\cal D}'(0,T), {\cal D}(0,T)} 
+ \int_0^{T} \int_{\Omega} \sum_{i, j =1}^3  \sigma_{ij} \frac{\partial \varphi_i}{\partial x_j} \chi \, dx dt
\\
\displaystyle  =
\int_0^ T \int_{\Omega} f \varphi \chi \, dx dt 
 - \int_0^ T \int_{\Omega} G_{0} \frac{\partial \zeta}{\partial t} \varphi \chi \, dx dt .
\end{array}
\end{eqnarray*}
Hence the regularity of $div(\sigma)$ is ``governed'' by  the regularity of $\displaystyle  {\partial  \widetilde{v} \over\partial t}$. Moreover, since $\sigma \in H^{-1} \bigl(0,T; \bigl( L^2(\Omega) \bigr)^{3 \times 3} \bigr)$, we may only expect $\sigma_n$ to belong to $H^{-1} \bigl(0,T; H^{-1/2} (\partial \Omega) \bigr)$. It follows that we will need to regularize $\sigma_n$ not only with respect to the space variable as in \cite{duvaut2, duvaut3, demkowicz} but also with respect to the time variable. Of course, in the case of an evolutionary problem, it will be a non sense to propose a convolution of $\sigma_n$ with respect to the time variable on $[0,T]$ and the most natural regularization seems to replace $\sigma_n$ by $\sigma_n^*$ and then to regularize $\sigma_n^*$ by a kind of truncated convolution on each time interval $[0,t]$, $t \in [0,T]$, leading to a space-time non-local friction law described by 
\begin{eqnarray*}
{\cal F}(x', t, \sigma) = k \int_0^t S(t-s) \bigl| \sigma_n^* (x', s) \bigr| \, ds
\end{eqnarray*}
for almost every $x' \in \Gamma_0$ and for all $t \in [0,T]$, where $S$ is a non negative smooth real function.

Let us emphasize that $S$ can also be interpreted as the kernel of some history-dependent shear stress threshold. Such kind of friction laws have been recently developped in the framework of solid mechanics (see \cite{sofonea} and the references therein for instance).

\bigskip

The outline of our paper is as follows. In the next section we consider the Tresca's case when the  mapping ${\cal F}$ is a given non-negative function of $(x',t)$ and does not depend on $\sigma$. We will establish an existence and uniqueness result for problem $(P)$ as well as some estimates of the solution, by using a Yosida's approximation of the slip condition (Theorem \ref{tresca-existence}). Then, under some compatibility assumptions on the initial velocity $v_0$, we prove additional regularity properties and sharp estimates for $\displaystyle \frac{\partial {\widetilde v}}{\partial t}$, $p$ and $\sigma$ (Theorem \ref{tresca-regularity} and Proposition \ref{proposition1}). Then, in Section \ref{coulomb} we consider the generalized Coulomb's friction case described by 
\begin{eqnarray*}
{\cal F}(x', t, \sigma) = {\cal F}^0 (x',t) + {\cal F}^{\sigma} (x',t) \int_0^t S(t-s) \bigl| \sigma_n^* (x', s) \bigr| \, ds
\end{eqnarray*}
for almost every $x' \in \Gamma_0$ and for all $t \in [0,T]$. If ${\cal F}^{\sigma} \equiv 0$ we recover the Tresca's friction case and when ${\cal F}^0 \equiv 0$ and ${\cal F}^{\sigma} \equiv k$, with $k>0$, we obtain the space-time non-local friction law introduced previously.

For this generalized Coulomb's friction law we prove an existence result by applying a successive approximation technique with respect to the shear stress threshold.


\renewcommand{\theequation}{3.\arabic{equation}}
 \setcounter{equation}{0}

\section{The Tresca's friction case} \label{tresca}

Let us assume from now on that ${\cal F}$ does not depend on its third argument, i.e.
\begin{eqnarray*} 
{\cal F }(x', t, \sigma) = \ell (x',t) \ \mbox{ on } \Gamma_0 \times (0, T), 
\end{eqnarray*}
 with
 \begin{eqnarray} \label{tresca.1}
 \quad \ell \in  L^{2}\bigl(0,T;  L^{2}(\Gamma_0; \br^+)  \bigr) 
  \end{eqnarray}
Let ${\bf H}$ be the closure in $\bigl( L^2 (\Omega) \bigr)^3$ of $\bigl\{ \varphi \in \bigl( {\cal D}(\Omega) \bigr)^3; div (\varphi) = 0 \bigr\}$.

\begin{theorem} \label{tresca-existence}
Let assumptions (\ref{mu_et_f})-(\ref{zeta})-(\ref{G_0})-(\ref{tresca.1}) hold. Then, for all ${\widetilde v}^0 \in {\bf H}$, problem $(P)$ admits an unique solution. Furthermore, $\displaystyle \frac{\partial \widetilde v}{\partial t} \in L^{2} \bigl(0,T; {\cal V}_{0div}' \bigr)$.
\end{theorem}

\begin{proof}
 For the sake of notational simplicity, let us denote  ${\bf X}$ instead of $X^3$ for any functional space $X$. For any $\varepsilon >0$ we consider the following approximate problem $(P_{\varepsilon})$

\noindent {\bf Problem $(P_{\varepsilon})$}
Find
\begin{eqnarray*}
\widetilde{v}_{\varepsilon}\in L^{2} \bigl(0,T;{\cal V}_{0div} \bigr)\cap L^{\infty}\bigl(0,T; {\bf L}^{2}(\Omega) \bigr),\  \, 
p_{\varepsilon} \in H^{-1}\bigl(0,T;L^{2}_{0}(\Omega) \bigr)
\end{eqnarray*}
such that, for all $ \varphi\in {\cal V}_{0}$ and for all $\chi\in {\cal D}(0,T)$, we have
\begin{equation}\label{NS-25bis}
\begin{array}{ll}
\displaystyle
\left\langle \frac{d}{dt} \left( \widetilde{v}_{\varepsilon}, \varphi \right) , \chi\right\rangle_{{\cal D}'(0,T), {\cal D}(0,T)}  
- \bigl\langle \bigl(p_{\varepsilon},div(\varphi) \bigr), \chi \bigr\rangle_{{\cal D}'(0,T), {\cal D}(0,T)} \\
\displaystyle +\int_0^T a(\widetilde{v}_{\varepsilon},\varphi\chi) \, dt 
+ \int_0^{T} \int_{\Gamma_0} \ell \frac{\widetilde{v}_{\varepsilon} \cdot \varphi}{\sqrt{\varepsilon^{2} + |\widetilde{v}_{\varepsilon}|^{2}}} \chi \,dx' dt
\displaystyle  =
\bigl\langle (f,\varphi ), \chi \bigr\rangle_{{\cal D}'(0,T), {\cal D}(0,T)} \\ 
\displaystyle - \int_0^T a(G_{0}\zeta ,\varphi \chi) \, dt 
\displaystyle
- \left\langle \left(G_{0}\frac{\partial \zeta}{\partial t},\varphi \right), \chi \right\rangle_{{\cal D}'(0,T), {\cal D}(0,T)}
\end{array}
\end{equation}
with the initial condition
\begin{eqnarray}\label{NS-25bisinit}
\widetilde{v}_{\varepsilon}(0, \cdot) = \widetilde{v}^{ 0}  .
\end{eqnarray}

As a first step we solve $(P_{\varepsilon})$ by using Galerkin's method. Recalling that ${\bf H} = \bigl\{ \varphi \in {\bf L}^2(\Omega); \ div(\varphi) =0 \ \mbox{ in $\Omega$}, \ u_n = 0 \ \mbox{ on $\partial \Omega$} \bigr\}$ (see \cite{temam} for instance) we infer that there exists a Hilbertian basis $(w_i)_{i \ge 1}$ of ${\bf H}$ which is orthonormal for the inner product of ${\bf L}^2(\Omega)$, and such that $(w_i)_{i \ge 1}$ is also a Hilbertian basis of ${\cal V}_{0 div}$ which is orthogonal for the inner product of ${\bf H}^1(\Omega)$.
Then, for all $m \ge 1$,  we look for a function $\widetilde{v}_{\varepsilon m}$ given by
\begin{eqnarray*}\label{NS34}
\widetilde{v}_{\varepsilon m}(t,x) = \sum_{j=1}^{m} g_{\varepsilon j} (t)w_{j}(x),\quad  \forall t\in(0,T),\   \forall  x\in\Omega,
\end{eqnarray*}
such that, for all $k \in \{1, \dots, m\}$, we have
\begin{equation}\label{NS35}
\begin{array}{ll}
\displaystyle
\left(\frac{\partial \widetilde{v}_{\varepsilon m}}{\partial t},w_{k}\right)
+a(\widetilde{v}_{\varepsilon m}, w_{k})  
+ \int_{\Gamma_0} \ell \frac{\widetilde{v}_{\varepsilon m} \cdot w_k}{\sqrt{ \varepsilon^2 + |\widetilde{v}_{\varepsilon m}|^2}} \, dx'
\\
\displaystyle =
(f,w_{k}) 
 - a(G_{0} \zeta ,w_{k})
  - \left( G_{0} \frac{\partial \zeta}{\partial t},w_{k} \right) 
\quad \mbox{\rm a.e. in } (0, T)
\end{array}
\end{equation}
with  the initial condition
\begin{eqnarray}\label{NS35-1}
\widetilde{v}_{\varepsilon m}(0, \cdot)=\widetilde{v}_{ m}^{ 0}
\end{eqnarray}
where   $\widetilde{v}_{ m}^{ 0}$ is  the orthogonal projection of $\widetilde{v}^{ 0}$ in $\bigl( L^{2}(\Omega) \bigr)^3$ on  ${\rm Span} \bigl\{w_{1}\ldots w_{m}\bigr\}$.

By using Caratheodory's theorem (see \cite{coddington}), we obtain that (\ref{NS35})-(\ref{NS35-1}) admits 
a unique maximal solution
$\widetilde{v}_{\varepsilon m} \in W^{1,2} (0, \tau_m; {\cal V}_{0 div})$, with $\tau_{m}\in (0, T]$.
 As usual we may establish some a priori estimates independent of $m$ and  $\varepsilon$
 which allow us to extend this solution to the whole interval $[0,T]$.

\begin{lemma}\label{lemma1}
Assume that (\ref{mu_et_f}), (\ref{zeta}), (\ref{G_0}), (\ref{tresca.1})  hold and that ${\widetilde v}^0 \in {\bf H}$.  Then, for all $m \ge 1$, the problem  (\ref{NS35})-(\ref{NS35-1}) admits a unique solution $\widetilde{v}_{\varepsilon m} \in W^{1,2} \bigl(0, T; {\cal V}_{0 div } \bigr)$ which satisfies the following estimates:
\begin{eqnarray}\label{NS36}
\|\widetilde{v}_{\varepsilon m} \|_{L^{\infty}(0,T; {\bf L}^{2}(\Omega))} \leq C_1
\end{eqnarray}
\begin{eqnarray}\label{NS37}
\|\widetilde{v}_{\varepsilon m} \|_{L^{2}(0,T; {\bf H}^1 (\Omega) )} \leq C_1
\end{eqnarray}
where $C_1$ is a positive constant independent of   $m$ and $\varepsilon$.
\end{lemma}


\begin{proof}
We multiply  equation (\ref{NS35}) by $g_{\varepsilon k}(t)$ and we add from $k=1$ to $m$. Observing that $\ell$ is a  non negative mapping we obtain
\begin{eqnarray*}
\begin{array}{ll}
\displaystyle{ \left(\frac{\partial \widetilde{v}_{\varepsilon m} }{\partial t},\widetilde{v}_{\varepsilon m} \right)
+ \int_{\Omega} 2 \mu D(\widetilde{v}_{\varepsilon m} ) : D(\widetilde{v}_{\varepsilon m}) \, dx } \\
\displaystyle{
\le  ( f ,\widetilde{v}_{\varepsilon m} )
- \int_{\Omega} 2 \mu D (G_{0} \zeta) : D ( \widetilde{v}_{\varepsilon m}  ) \, dx
- \left( G_{0} \frac{\partial \zeta}{\partial t}, \widetilde{v}_{\varepsilon m}  \right) } 
\quad \mbox{\rm a.e. in } (0, \tau_m) .
\end{array}
\end{eqnarray*}
We integrate from  $0$ to $s$, with $0<s<\tau_m$. Then, by using (\ref{coercif}) and Young's inequalities, we obtain
\begin{eqnarray*}
\begin{array}{ll}
\displaystyle{\frac{1}{2}\|\widetilde{v}_{\varepsilon m} (s)\|^{2}_{{\bf L}^{2}(\Omega)}
+ \frac{\alpha}{2}\int_{0}^{s}\|\widetilde{v}_{\varepsilon m} \|^{2}_{{\bf H}^{1}(\Omega)}\,dt
\leq\frac{1   }{2}\|\widetilde{v}_{\varepsilon m} (0)\|^{2}_{{\bf L}^{2}(\Omega)} +\frac{1}{2}\int_{0}^{s}\|f\|_{{\bf L}^{2}(\Omega)}^{2}\,dt }\\
\displaystyle{ +\frac{2 \mu^{2}}{\alpha}\|G_{0}\|_{{\bf H}^{1}(\Omega)}^{2}\int_{0}^{s}|\zeta|^{2}\,dt+\frac{1}{2}\|G_{0}\|_{{\bf L}^{2}(\Omega)}^{2}\int_{0}^{s}\left|\frac{\partial \zeta}{\partial t}\right|^{2}\,dt+ \int_{0}^{s}\|\widetilde{v}_{\varepsilon m} \|_{{\bf L}^{2}(\Omega)}^{2}\,dt. }
\end{array}
\end{eqnarray*}

Recalling that $\widetilde{v}_{\varepsilon m } (0)$ is defined as the orthogonal projection of $\widetilde{v}^{ 0} $ in ${\bf L}^{2}(\Omega) $ on  ${\rm Span} \bigl\{w_{1} , \ldots ,  w_{m}\bigr\}$, we have
\begin{eqnarray}
\label{NS36-2}
\displaystyle \frac{1}{2}\|\widetilde{v}_{\varepsilon m} (s)\|^{2}_{{\bf L}^{2}(\Omega)}
+ \frac{\alpha}{2}\int_{0}^{s}\|\widetilde{v}_{\varepsilon m} \|^{2}_{{\bf H}^{1}(\Omega)}\,dt
\leq
 C'_{1} 
 \displaystyle + \int_{0}^{s}\|\widetilde{v}_{\varepsilon m} \|_{{\bf L}^{2}(\Omega)}^{2}\,dt,
\end{eqnarray}
where $C'_{1}$ is given by 
\begin{eqnarray*}
\begin{array}{ll}
\displaystyle C'_{1}=\frac{1}{2} \|\widetilde{v}^0\|^{2}_{{\bf L}^{2}(\Omega)} +\frac{1}{2}\int_{0}^{T}\|f\|_{{\bf L}^{2}(\Omega)}^{2}\,dt +
\frac{2 \mu^{2}}{\alpha}\|G_{0}\|_{{\bf H}^{1}(\Omega)}^{2}\int_{0}^{T}|\zeta|^{2}\,dt \\
\displaystyle 
+ \frac{1}{2}\|G_{0}\|_{{\bf L}^{2}(\Omega)}^{2}\int_{0}^{T}\left|\frac{\partial \zeta}{\partial t}\right|^{2}\,dt
\end{array}
\end{eqnarray*}

With Gr\"onwall's lemma, we get
\begin{eqnarray}\label{estimL2}
\|\widetilde{v}_{\varepsilon m} (s)\|^{2}_{{\bf L}^{2}(\Omega)} \leq 2C'_{1}\exp\left(2s \right)  \le 2C'_{1}\exp\left(2 T \right) \quad \forall s \in [0, \tau_m).
\end{eqnarray}
By definition of the maximal solution, we may conclude  that $\tau_m=T$ and 
 (\ref{NS36}) follows from (\ref{estimL2}). By inserting (\ref{estimL2}) in (\ref{NS36-2}) with $s=T$, we obtain (\ref{NS37}).

\end{proof}

As a corollary, we obtain also a uniform  estimate of $\displaystyle \frac{\partial \widetilde{v}_{\varepsilon m}}{\partial t} $ in $L^{2}(0,\tau;{\cal V}_{0 div}')$.

\begin{lemma} \label{lemma2}
Assume that (\ref{mu_et_f}), (\ref{zeta}),  (\ref{G_0}),  (\ref{tresca.1})  hold and that ${\widetilde v}^0 \in {\bf H}$.  Then there exists a positive real number $C_2$, independent of $m$ and $\varepsilon$, such that
\begin{eqnarray}\label{eq47}
\left\|\frac{\partial \widetilde{v}_{\varepsilon m}}{\partial t} \right\|_{L^{2}(0,T;{\cal V}_{0 div}')}\leq C_2.
\end{eqnarray}
\end{lemma}

\begin{proof}
Let $\varphi \in {\cal V}_{0 div}$. For all $m \ge 1$, we define $\varphi_{m}$ as the orthogonal projection with respect to the  inner product of
${\bf H}^{1}(\Omega)$
of $\varphi$ on   ${\rm Span} \bigl\{w_{1},\ldots,w_{m} \bigr\}$.
With (\ref{NS35}) we get
\begin{eqnarray}
\label{v'43}
\begin{array}{ll}
\displaystyle{   \left(\frac{ \partial \widetilde{v}_{\varepsilon m} }{\partial t} ,\varphi_{m}\right)
= 
- \int_{\Omega} 2 \mu  D(\widetilde{v}_{\varepsilon m} ) : D(\varphi_m) \, dx
- \int_{\Gamma_0} \ell \frac{\widetilde{v}_{\varepsilon m}  \cdot \varphi_m}{\sqrt{\varepsilon^2 + |\widetilde{v}_{\varepsilon m}|^2}} \, dx'}
\\
\displaystyle{ +
(f, \varphi_m)
- \int_{\Omega}  2 \mu D(G_0 \zeta): D (\varphi_m) \, dx
  - \left( G_{0} \frac{\partial \zeta}{\partial t}, \varphi_m \right) 
\quad \mbox{\rm a.e. in } (0, T).}
\end{array}
\end{eqnarray}
We estimate all the terms in the right hand side of the previous equality, we obtain
\begin{eqnarray*}
\begin{array}{ll}
\displaystyle{\left|\left( \frac{\partial \widetilde{v}_{\varepsilon m} }{\partial t},\varphi_{m}\right)\right|
\leq
2 \mu \| \widetilde{v}_{\varepsilon m} \|_{{\bf H}^{1}(\Omega)}\| \varphi_{m}\|_{{\bf H}^{1}(\Omega)}
+\|\ell\|_{{\bf L}^{2}(\Gamma_0)}\|\varphi_{m}\|_{{\bf L}^{2}(\Gamma_0)}
}
\\
\displaystyle{
+ \|f\|_{{\bf L}^{2}(\Omega)}\|\varphi_{m}\|_{{\bf L}^{2}(\Omega)}+ 2 \mu |\zeta|\|G_{0}\|_{{\bf H}^{1}(\Omega)}\|\varphi_{m}\|_{{\bf H}^{1}(\Omega)} } \\
\displaystyle{
+   \left|\frac{\partial\zeta}{\partial t}\right|\|G_{0}\|_{{\bf L}^{2}(\Omega)}\|\varphi_{m}\|_{{\bf L}^{2}(\Omega)} }
 \quad \mbox{\rm a.e. in } (0, T).
\end{array}
\end{eqnarray*}

As $(w_{i})_{i\geq 1}$ is an orthogonal family of  ${\bf L}^{2}(\Omega)$ and $\varphi_{m}$ is the orthogonal
projection with respect to the  inner product of
${\bf H}^{1}(\Omega)$
of $\varphi$ on   ${\rm Span} \bigl\{w_{1},\ldots,w_{m} \bigr\}$,
 we have $\|\varphi_{m}\|_{{\bf H}^{1}(\Omega)}\leq \|\varphi\|_{{\bf H}^{1}(\Omega)}$ and
\begin{eqnarray*}
\left( \frac{ \widetilde{v}_{\varepsilon m}}{\partial t} ,\varphi_{m}\right)=\left( \frac{\partial \widetilde{v}_{\varepsilon m}}{\partial t} ,\varphi_{k} \right) \quad \forall k \ge m.
\end{eqnarray*}
Moreover $(w_{i})_{i\geq 1}$ is a Hilbertian basis of ${\cal V}_{0 div}$, so the sequence $(\varphi_k)_{k \ge 1}$ converges strongly to $\varphi$ in ${\bf H}^1(\Omega)$ and we get
\begin{eqnarray*}
\left( \frac{ \widetilde{v}_{\varepsilon m} }{\partial t} ,\varphi_{m}\right)=\left( \frac{\partial \widetilde{v}_{\varepsilon m} }{\partial t} ,\varphi \right) .
\end{eqnarray*}
Then, we obtain
\begin{eqnarray*}
\begin{array}{ll}
\displaystyle{
\left|\left( \frac{\partial \widetilde{v}_{\varepsilon m} }{\partial t},\varphi\right)\right|
\leq
2 \mu \|\widetilde{v}_{\varepsilon m}\|_{{\bf H}^{1}(\Omega)}\|\varphi\|_{{\bf H}^{1}(\Omega)}
+  c (\gamma_0) \| \ell\|_{{\bf L}^{2}(\Gamma_0)}
\|\varphi\|_{{\bf H}^{1}(\Omega)}
} \\
\displaystyle{+\left(\|f\|_{{\bf L}^{2}(\Omega)}+ 2 \mu |\zeta|\|G_{0}\|_{{\bf H}^{1}(\Omega)}+   \left|\frac{\partial\zeta}{\partial t}\right|\|G_{0}\|_{{\bf L}^{2}(\Omega)}\right)\|\varphi\|_{{\bf H}^{1}(\Omega)}  \ \mbox{\rm a.e. in $(0, T)$,} }
\end{array}
\end{eqnarray*}
where $ c (\gamma_0)$ is the norm of the trace operator $\gamma_0: {\bf H}^{1}(\Omega) \to {\bf L}^2 (\Gamma_0)$.
Hence
\begin{eqnarray*}
\begin{array}{ll}
\displaystyle{
\left\| \frac{\partial \widetilde{v}_{\varepsilon m} }{\partial t} \right\|_{{\cal V}_{0 div}'}
 \leq
2 \mu \|\widetilde{v}_{\varepsilon m}\|_{{\bf H}^{1}(\Omega)} 
+  c (\gamma_0)  \| \ell \|_{{\bf L}^{2}(\Gamma_0)}+\|f\|_{{\bf L}^{2}(\Omega)}
}
\\
\displaystyle{
 + 2 \mu |\zeta|\|G_{0}\|_{{\bf H}^{1}(\Omega)}+   \left|\frac{\partial\zeta}{\partial t}\right|\|G_{0}\|_{{\bf L}^{2}(\Omega)}  \quad \mbox{\rm a.e. in } (0, T)}
\end{array}
\end{eqnarray*}
and the conclusion follows from the estimates of Lemma \ref{lemma1}.  

\end{proof}

Now we can pass to the limit as $m$ tends to $+ \infty$. Indeed there exists a subsequence of $(\widetilde{v}_{\varepsilon m})_{m \ge 1}$, still denoted $(\widetilde{v}_{\varepsilon m})_{m \ge 1}$, such that
\begin{eqnarray*}\label{conver2}
\displaystyle{
\widetilde{v}_{\varepsilon m} \rightharpoonup \widetilde{v}_{\varepsilon} \quad\mbox{weakly star in $L^{\infty}\bigl(0,T;{\bf L}^{2}(\Omega) \bigr)$ and weakly in $ L^{2}(0,T;{\cal V}_{0 div}) $}}
\end{eqnarray*}
and 
\begin{eqnarray*}\label{conver3}
\frac{\partial \widetilde{v}_{\varepsilon m}}{\partial t} \rightharpoonup \frac{\partial \widetilde{v}_{\varepsilon }}{\partial t} 
\quad\mbox{weakly in $ L^{2}(0,T;{\cal V}_{0 div}') $.}
\end{eqnarray*}
By using Aubin's lemma we infer that, possibly extracting another subsequence, 
\begin{eqnarray*}
\widetilde{v}_{\varepsilon m} \rightarrow \widetilde{v}_{\varepsilon} \quad\mbox{strongly in $L^{2}(0,T;{\bf H}^{s}(\Omega))$}
\end{eqnarray*}
with $\frac{1}{2} < s <1$ and thus
\begin{eqnarray} \label{conver6}
\widetilde{v}_{\varepsilon m} \rightarrow \widetilde{v}_{\varepsilon} \quad\mbox{strongly in $L^{2}(0,T;{\bf L}^{2}(\Gamma_0))$ and a.e. on $\Gamma_0 \times (0,T)$.}
\end{eqnarray}

Let $\chi\in L^2(0,T )$ and $\varphi\in {\cal V}_{0 div }$. For all $m \ge 1$ we define again $\varphi_{m}$ as the orthogonal projection of $\varphi$ with respect to the  inner product of
${\bf H}^{1}(\Omega)$
 on   ${\rm Span} \bigl\{w_{1},\ldots,w_{m} \bigr\}$. We multiply (\ref{v'43}) by $\chi$, we integrate on $(0,T)$ and we pass to the limit as $m$ tends to $+ \infty$.  We get 
\begin{eqnarray} \label{v'43bis}
\begin{array}{ll}
\displaystyle \int_0^T  \left\langle \frac{ \partial \widetilde{v}_{\varepsilon } }{\partial t} ,\varphi \right\rangle_{{\cal V}_{0 div}', {\cal V}_{0 div}}  \chi  \, dt 
 + \int_0^T \int_{\Omega} 2 \mu  D(\widetilde{v}_{\varepsilon } ) : D(\varphi  \chi ) \, dx dt \\
 \displaystyle
+ \int_0^T \int_{\Gamma_0} \ell \frac{\widetilde{v}_{\varepsilon }  \cdot \varphi}{\sqrt{\varepsilon^2 + |\widetilde{v}_{\varepsilon }|^2}} \chi \, dx' dt 
=
\int_0^T (f, \varphi \chi ) \, dt 
\\
\displaystyle{ 
- \int_0^T  \int_{\Omega}  2 \mu D(G_0 \zeta): D (\varphi \chi ) \, dx dt
  - \int_0^T \left( G_{0} \frac{\partial \zeta}{\partial t}, \varphi \chi  \right)  \, dt.
}
\end{array}
\end{eqnarray}
Furthermore, by using Simon's lemma and  possibly extracting another subsequence,   we have
\begin{eqnarray*}\label{lemsimonc3}
\widetilde{v}_{\varepsilon m}\rightarrow \widetilde{v}_{\varepsilon}\quad\mbox{strongly in ${\cal C}^{0}\bigl( [0,T]; {\cal H} \bigr)$,}
\end{eqnarray*}
for any Banach space ${\cal H}$  such that ${\bf L}^{2}(\Omega)\subset {\cal H} \subset {\cal V}_{0 div}'$ with continuous injections and compact embedding of ${\bf L}^{2}(\Omega)$ into ${\cal H}$. Recalling that 
\begin{eqnarray*}
\widetilde{v}_{\varepsilon m} (0)  = \widetilde{v}^{0}_{ m} \rightarrow \widetilde{v}^{ 0}  \quad \mbox{strongly in ${\bf L}^2 (\Omega)$}
\end{eqnarray*}
we infer that $\widetilde{v}_{\varepsilon} (0) = \widetilde{v}^{ 0}$.

\bigskip

Finally, using De Rham's theorem, we obtain that there exists $ p_{\varepsilon}$ such that $(\widetilde v_{\varepsilon}, p_{\varepsilon})$ is a solution of problem $(P_{\varepsilon})$.
Indeed, possibly modifying $\widetilde v_{\varepsilon}$ on a negligible subset of $[0,T]$, we have $\widetilde v_{\varepsilon} \in C^0 \bigl( [0,T]; { \bf H} \bigr)$. Hence 
for all $t \in [0,T]$, we may define $F_{\varepsilon} \in C^0 \bigl( [0,T]; {\bf H}^{-1}(\Omega) \bigr)$ by
\begin{eqnarray*}
\begin{array}{ll}
\displaystyle \bigl\langle F_{\varepsilon} (t) , \varphi \bigr\rangle_{{\bf H}^{-1}(\Omega), {\bf H}^1_0(\Omega) }
= - \bigl( \widetilde{v}_{\varepsilon} (t) - \widetilde{v}^0, \varphi \bigr) - \int_0^t \int_{\Omega} 2 \mu D( \widetilde{v}_{\varepsilon} + G_0 \zeta) : D(\varphi) \, dx dt \\
\displaystyle + \int_0^t  ( f, \varphi) - \int_0^t \left(G_0 \frac{ \partial \zeta}{\partial t}, \varphi \right)  \, dt \quad \forall \varphi \in {\bf H}^1_0(\Omega), \ \forall t \in [0,T].
\end{array}
\end{eqnarray*}
With (\ref{v'43bis}) we obtain that
\begin{eqnarray*}
\bigl\langle F_{\varepsilon} (t) , \varphi \bigr\rangle_{{\bf H}^{-1}(\Omega) , {\bf H}^1_0(\Omega) } = 0 
\end{eqnarray*}
for all $\varphi \in {\bf H}^1_0 (\Omega)$ such that $ div (\varphi) = 0$ and we infer that there exists ${ \pi}_{\varepsilon}  \in C^0 \bigl( [0,T];  L^2_0(\Omega) \bigr)$ such that 
\begin{eqnarray} \label{pression.1}
\bigl\langle F_{\varepsilon} (t) , \varphi \bigr\rangle_{{\bf H}^{-1} (\Omega) , {\bf H}^1_0(\Omega) } = \bigl\langle \nabla { \pi}_{\varepsilon} (t) , \varphi \bigr\rangle_{{\cal D}' (\Omega), {\cal D}(\Omega)} \quad \forall \varphi \in {\cal D}(\Omega), \  \forall t \in [0,T].
\end{eqnarray}
Let us denote now by $p_{\varepsilon}$ the time derivative of ${ \pi}_{\varepsilon}$ in the distribution sense. From (\ref{pression.1}) we obtain
 \begin{eqnarray}  \label{pression.2}
\begin{array}{ll}
\displaystyle \frac{d }{d t}( \widetilde{v}_{\varepsilon} , \varphi ) + a ( \widetilde{v}_{\varepsilon} + G_0 \zeta ,\varphi) + 
  \bigl\langle \nabla { p}_{\varepsilon}  , \varphi \bigr\rangle_{{\cal D}' (\Omega), {\cal D}(\Omega)} \\
\displaystyle =  ( f, \varphi) -  \left(G_0 \frac{ \partial \zeta}{\partial t}, \varphi \right)  \, dt \quad \forall \varphi \in {\bf {\cal D}}(\Omega)
\end{array}
\end{eqnarray}
in ${\cal D}'(0,T)$.

\begin{lemma} \label{lemma3}
We have $ p_{\varepsilon} \in H^{-1} (0,T;L^{2}_0(\Omega))$ and  there exists a constant $C_3$ independent of $\varepsilon$ such that
\begin{eqnarray*}
\| p_{\varepsilon} \|_{H^{-1} (0,T;L^{2}(\Omega))} \le C_3.
\end{eqnarray*}
\end{lemma}

\begin{proof}
With (\ref{pression.2}) we get
 \begin{eqnarray*} 
\begin{array}{ll}
\displaystyle  \bigl\langle \bigl( { p}_{\varepsilon}  , div( \varphi ) \bigr),  \chi \bigr\rangle_{{\cal D}'(0,T), {\cal D}(0,T) } = 
- \int_0^T ( \widetilde{v}_{\varepsilon} , \varphi ) \chi'  \, dt  + \int_0^T a ( \widetilde{v}_{\varepsilon} + G_0 \zeta ,\varphi) \chi \, dt  
   \\
\displaystyle - \int_0^T  ( f, \varphi)\chi \, dt  +  \int_0^T \left(G_0 \frac{ \partial \zeta}{\partial t}, \varphi \right)  \chi \, dt \quad \forall \varphi \in {\bf {\cal D}} (\Omega), \ \forall \chi \in  {\cal D} (0,T).
\end{array}
\end{eqnarray*}
By density the same equality is still valid for all $\varphi \in {\bf H}^1_0 (\Omega)$. Now let $\widetilde w \in L^2(\Omega)$ and
\begin{eqnarray*}
w= \widetilde w - \frac{1}{|\Omega|} \int_{\Omega} \widetilde w \, dx .
\end{eqnarray*}
By construction we have $w \in L^2_0(\Omega)$ and $\|w\|_{L^2(\Omega)} \le \| \tilde w\|_{L^2(\Omega)}$. Moreover there exists a linear continuous operator $P: L^2_0(\Omega) \to {\bf H}^1_0 (\Omega)$ such that  
\begin{eqnarray*}
P(w) = \varphi \in {\bf H}^1_0 (\Omega), \quad div (\varphi) = w \quad \forall w \in L^2_0(\Omega)
\end{eqnarray*}
 (see \cite{girault}). 
It follows that 
 \begin{eqnarray*} 
\begin{array}{ll}
\displaystyle 
\left| \bigl\langle ( { p}_{\varepsilon}  , \widetilde w ),  \chi \bigr\rangle_{{\cal D}'(0,T), {\cal D}(0,T) } \right|
 = \left| \bigl\langle ( { p}_{\varepsilon}  , w ) ,  \chi \bigr\rangle_{{\cal D}'(0,T), {\cal D}(0,T) } \right| \\
 \displaystyle 
 =  \left| \bigl\langle \bigl( { p}_{\varepsilon}  , div( \varphi ) \bigr),  \chi \bigr\rangle_{{\cal D}'(0,T), {\cal D}(0,T) } \right|
\\
\displaystyle  \le \| \widetilde{v}_{\varepsilon} \|_{L^{2} (0,T; {\bf L}^{2} (\Omega))} \| P(w) \chi' \|_{L^{2} (0,T;{\bf L}^{2}(\Omega))} \\
\displaystyle + 2 \mu  \bigl( \| \widetilde{v}_{\varepsilon} \|_{L^{2} (0,T;{\bf H}^1(\Omega) )} + \sqrt{T} \|G_0\|_{{\bf H}^1(\Omega)} \| \zeta \|_{C^0([0,T]} \bigr) \| P(w) \chi \|_{L^{2} (0,T; {\bf H}^{1}(\Omega))} \\
\displaystyle  + \left( \|f\|_{L^{2} (0,T;{\bf L}^{2}(\Omega))} + \sqrt{T} \|G_0\|_{{\bf L}^2(\Omega)} \left\| \frac{\partial \zeta}{\partial t} \right\|_{C^0([0,T])} \right) \| P(w) \chi \|_{L^{2} (0,T; {\bf L}^{2}(\Omega))} \\
\displaystyle \forall \chi \in  {\cal D} (0,T).
\end{array}
\end{eqnarray*}
Since the estimates obtained in Lemma \ref{lemma1}  are independent of $m$ and $\varepsilon$, we infer that the sequence $(\widetilde{v}_{\varepsilon })_{\varepsilon  >0}$ is  bounded in $ L^{2}(0,T;{\cal V}_{0 div}) \cap L^{\infty}\bigl(0,\tau;{\bf L}^{2}(\Omega) \bigr) $. Finally, by using the continuity of the operator $P$ and the density of ${\cal D}(0,T) \otimes L^2(\Omega)$ into $H^1_0 \bigl( 0,T; L^2(\Omega) \bigr)$, we may conclude.

\end{proof}

It follows that $\sigma_{\varepsilon} = - p_{\varepsilon} {\rm Id} + 2 \mu D(\widetilde{v}_{\varepsilon} + G_0 \zeta) $ belongs to $H^{-1} \bigl(0,T; \bigl(L^{2}(\Omega)\bigr)^{3 \times 3} \bigr)$ and is bounded in $H^{-1} \bigl(0,T; \bigl(L^{2}(\Omega)\bigr)^{3 \times 3} \bigr)$ uniformly with respect to   $\varepsilon$. Moreover, we have
\begin{eqnarray*}
&&\left\langle {d \over d t} \left( \widetilde{v}_{\varepsilon}, \varphi \right) , \chi\right\rangle_{{\cal D}'(0,T), {\cal D}(0,T)} 
+ \int_0^T a(\widetilde{v}_{\varepsilon}+ G_0 \zeta,\varphi \chi) \, dt  \nonumber\\
&&
- \bigl\langle \bigl(p_{\varepsilon} , div(\varphi)  \bigr),  \chi \bigr\rangle_{{\cal D}'(0,T), {\cal D}(0,T)} 
  =
\bigl\langle (f,\varphi ),  \chi \bigr\rangle_{{\cal D}'(0,T), {\cal D}(0,T)} 
\nonumber\\
&&
 - \left\langle \left(G_{0} \frac{\partial \zeta}{\partial t},\varphi \right),
  \chi \right\rangle_{{\cal D}'(0,T), {\cal D}(0,T)}  
 \quad \forall \varphi \in {\bf{\cal D}} (\Omega), \  \forall \chi \in  {\cal D} (0,T)
\end{eqnarray*}
i.e.
 \begin{eqnarray} \label{pression.3bis}
\begin{array}{ll}
\displaystyle  
- \int_0^T ( \widetilde{v}_{\varepsilon} , \varphi ) \chi'  \, dt  - \int_0^T  \bigl\langle div (\sigma_{\varepsilon}), \varphi \bigr\rangle_{{\bf{ \cal D}}' (\Omega), {\bf {\cal D}} (\Omega) } \chi \, dt  
   \\
\displaystyle = \int_0^T  ( f, \varphi)\chi \, dt  -  \int_0^T \left(G_0 \frac{ \partial \zeta}{\partial t}, \varphi \right)  \chi \, dt \quad \forall \varphi \in {\bf{\cal D}} (\Omega), \  \forall \chi \in  {\cal D} (0,T)
\end{array}
\end{eqnarray}
and we infer that $div(\sigma_{\varepsilon})$ belongs to $H^{-1} \bigl(0,T; {\bf L}^{2}(\Omega) \bigr)$ and is bounded in $H^{-1} \bigl(0,T; {\bf L}^{2}(\Omega) \bigr)$ uniformly with respect to   $\varepsilon$.


\bigskip

Let us consider now $\varphi \in {\cal V}_0$ in (\ref{pression.3bis}). 
Recalling that  $\sigma_{\varepsilon}$ and $div(\sigma_{\varepsilon})$ belong to $H^{-1} \bigl(0,T; \bigl(L^{2}(\Omega)\bigr)^{3 \times 3} \bigr)$ and 
$ H^{-1} \bigl( 0,T; {\bf L}^{2}(\Omega) \bigr)$ respectively, we may apply Green's formula (see \cite{girault}) and we get 
 \begin{eqnarray*} \label{pression.3}
\begin{array}{ll}
\displaystyle  
- \int_0^T ( \widetilde{v}_{\varepsilon} , \varphi ) \chi'  \, dt  +   \int_0^T a ( \widetilde{v}_{\varepsilon} + G_0 \zeta ,\varphi) \chi \, dt 
 - \int_0^T \bigl( { p}_{\varepsilon}  , div( \varphi ) \bigr) \chi \, dt \\
 \displaystyle 
 - \int_0^T \int_{\Gamma_0} \sum_{i,j =1}^3 \sigma_{i j \varepsilon} n_j \varphi_i \chi \, dx' dt \\
\displaystyle  = \int_0^T  ( f, \varphi)\chi \, dt 
-  \int_0^T \left(G_0 \frac{ \partial \zeta}{\partial t}, \varphi \right)  \chi \, dt \quad  \forall \chi \in  {\cal D} (0,T).
\end{array}
\end{eqnarray*}
But for any $\varphi \in {\cal V}_0$ we have
\begin{eqnarray*}
\int_{\partial \Omega} \varphi \cdot n \, d \gamma = 0
\end{eqnarray*}
and there exists ${\widetilde \varphi} \in {\bf H}^1 (\Omega)$ such that
\begin{eqnarray*}
{\widetilde \varphi} = \varphi \ \mbox{on $\partial \Omega$,} \quad div({\widetilde \varphi}) = 0 \ \mbox{ in $\Omega$}
\end{eqnarray*}
since $\Omega$ is connected (see \cite{girault}). Then we get
 \begin{eqnarray*} \label{pression.3}
\begin{array}{ll}
\displaystyle  
- \int_0^T ( \widetilde{v}_{\varepsilon} , {\widetilde \varphi} ) \chi'  \, dt  +   \int_0^T a ( \widetilde{v}_{\varepsilon} + G_0 \zeta , {\widetilde \varphi}) \chi \, dt 
 - \int_0^T \int_{\Gamma_0} \sum_{i,j =1}^3 \sigma_{i j \varepsilon } n_j {\widetilde \varphi}_i \chi \, dx' dt
 \\
= 
\displaystyle  \int_0^T  ( f, {\widetilde \varphi})\chi \, dt  -  \int_0^T \left(G_0 \frac{ \partial \zeta}{\partial t}, {\widetilde \varphi} \right)  \chi \, dt \quad  \forall \chi \in  {\cal D} (0,T).
\end{array}
\end{eqnarray*}
By comparing with (\ref{v'43bis}), we obtain
\begin{eqnarray*}
 - \int_0^T \int_{\Gamma_0} \sum_{i,j =1}^3 \sigma_{i j \varepsilon } n_j {\widetilde \varphi}_i \chi \, dx' dt 
 = \int_0^T \int_{\Gamma_0} \ell  \frac{\widetilde{v}_{\varepsilon} \cdot {\widetilde \varphi}}{\sqrt{\varepsilon^{2} + |\widetilde{v}_{\varepsilon}|^{2}}} \chi \,dx' dt \quad \forall \chi \in {\cal D}(0,T).
 \end{eqnarray*}
 Owing that ${\widetilde \varphi} = \varphi$ on $\Gamma_0$ we infer that
 \begin{eqnarray*}
 - \int_0^T \int_{\Gamma_0} \sum_{i,j =1}^3 \sigma_{i j \varepsilon} n_j { \varphi}_i \chi \, dx' dt 
 = \int_0^T \int_{\Gamma_0} \ell  \frac{\widetilde{v}_{\varepsilon} \cdot { \varphi}}{\sqrt{\varepsilon^{2} + |\widetilde{v}_{\varepsilon}|^{2}}} \chi \,dx' dt \quad \forall \chi \in {\cal D}(0,T)
 \end{eqnarray*}
 and finally $(\widetilde v_{\varepsilon}, p_{\varepsilon})$ is a solution of problem $(P_{\varepsilon})$.

%


\bigskip

Now, observing that 
\begin{eqnarray*}
\displaystyle \int_0^{T} \int_{\Gamma_0} \ell \bigl( \sqrt{ \varepsilon^2 + |\widetilde v_{\varepsilon} + \varphi \chi|^2} 
- \sqrt{ \varepsilon^2 + |\widetilde v_{\varepsilon} |^2} \bigr) \, dx' dt 
\displaystyle \ge \int_0^T \int_{\Gamma_0} \ell \frac{\widetilde{v}_{\varepsilon} \cdot \varphi}{\sqrt{\varepsilon^{2} + |\widetilde{v}_{\varepsilon}|^{2}}} \chi \,dx' dt
\end{eqnarray*}
we get from (\ref{NS-25bis}) the following variational inequality
\begin{equation} \label{NS-25ter}
\begin{array}{ll}
\displaystyle
\left\langle \frac{d}{dt} \left( \widetilde{v}_{\varepsilon}, \varphi \right) , \chi\right\rangle_{{\cal D}'(0,T), {\cal D}(0,T)} 
- \bigl\langle \bigl(p_{\varepsilon},div(\varphi) \bigr), \chi \bigr\rangle_{{\cal D}'(0,T), {\cal D}(0,T)}
+ \int_0^T a(\widetilde{v}_{\varepsilon},\varphi\chi)  \, dt \\
\displaystyle
+ \int_0^{T} \int_{\Gamma_0} \ell  \sqrt{ \varepsilon^2 + |\widetilde v_{\varepsilon} + \varphi \chi|^2} \, dx' dt
-\int_0^{T} \int_{\Gamma_0} \ell  \sqrt{ \varepsilon^2 + |\widetilde v_{\varepsilon} |^2} \, dx' dt \\
\displaystyle  \ge
\bigl\langle (f,\varphi ), \chi \bigr\rangle_{{\cal D}'(0,T), {\cal D}(0,T)} - \int_0^T a (G_{0}\zeta ,\varphi \chi) \, dt \\
\displaystyle - \left\langle \left(G_{0}\frac{\partial \zeta}{\partial t},\varphi \right), \chi \right\rangle_{{\cal D}'(0,T), {\cal D}(0,T)}
 \quad \forall \varphi\in {\cal V}_{0}, \  \forall \chi\in {\cal D}(0,T).
\end{array}
\end{equation}

Let us pass now to the limit as $\varepsilon$ tends to zero. Since the estimates obtained in Lemma \ref{lemma1}, Lemma \ref{lemma2} and Lemma \ref{lemma3} are independent of $m$ and $\varepsilon$, we infer that the sequences $(\widetilde{v}_{\varepsilon })_{\varepsilon  >0}$, $\displaystyle \left(\frac{\partial \widetilde{v}_{\varepsilon }}{\partial t} \right)_{\varepsilon  >0}$ and $(p_{\varepsilon })_{\varepsilon  >0}$ are bounded in $ L^{2}(0,T;{\cal V}_{0 div}) \cap L^{\infty}\bigl(0,T;{\bf L}^{2}(\Omega) \bigr) $,  $ L^{2}(0,T;{\cal V}_{0 div}')$ and $H^{-1} \bigl( 0,T; L^2_0 (\Omega) \bigr)$ respectively. Hence  we have the following the convergence results  
\begin{eqnarray*}
\displaystyle{
\widetilde{v}_{\varepsilon } \rightharpoonup \widetilde{v} \quad\mbox{weakly star in $L^{\infty}\bigl(0,T;{\bf L}^{2}(\Omega) \bigr)$ and weakly in $ L^{2}(0,T;{\cal V}_{0 div})$,}}
\end{eqnarray*}
\begin{eqnarray*}
\displaystyle{
\frac{\partial \widetilde{v}_{\varepsilon }}{\partial t} \rightharpoonup \frac{\partial \widetilde{v}}{\partial t} \quad\mbox{weakly in $ L^{2}(0,T;{\cal V}_{0 div}') $, }}
\end{eqnarray*}
\begin{eqnarray*}
\displaystyle{
\widetilde{p}_{\varepsilon } \rightharpoonup \widetilde{p}  \quad\mbox{weakly  in $H^{-1}\bigl(0,T;{ L}^{2}_0(\Omega) \bigr)$.}}
\end{eqnarray*}
Moreover, possibly extracting another subsequence, we have
\begin{eqnarray*}
\widetilde{v}_{\varepsilon} \rightarrow \widetilde{v} \quad\mbox{strongly in $ L^{2} \bigl(0,T;{\bf L}^{2}(\Gamma_0) \bigr)$,}
\end{eqnarray*}
and
\begin{eqnarray*}
\widetilde{v}_{\varepsilon } \rightarrow \widetilde{v}  \quad\mbox{strongly in ${\cal C}^{0}\bigl( [0,T]; {\cal H} \bigr)$,}
\end{eqnarray*}
for any Banach space ${\cal H}$  such that ${\bf L}^{2}(\Omega)\subset {\cal H} \subset {\cal V}_{0 div}'$ with continuous injections and compact  embedding of ${\bf L}^{2}(\Omega)$ into ${\cal H}$.
Then we may use the same arguments as previously to pass to the limit as $\varepsilon$ tends to zero in problem (\ref{NS-25ter}). Indeed, for the boundary term, we may apply the following property
\begin{eqnarray*}
\begin{array}{ll}
\displaystyle \left| \int_0^{T} \int_{\Gamma_0} \ell  \sqrt{ \varepsilon^2 + |\widetilde v_{\varepsilon} + \varphi \chi|^2}   \, dx' dt 
 - \int_0^{T} \int_{\Gamma_0} \ell |\widetilde v  + \varphi \chi| \, dx' dt \right| \\
 \displaystyle \le \|\ell\|_{L^2(0, T; {\bf L}^2(\Gamma_0))} \Bigl( \| \widetilde v_{\varepsilon} - \widetilde v \|_{L^2(0, T; {\bf L}^2(\Gamma_0))} + \varepsilon \sqrt{T {\rm meas}(\Gamma_0)}  \Bigr)
 \end{array}
 \end{eqnarray*}
for all $\varphi\in {\cal V}_{0}$ and for all $ \chi\in {\cal D}(0,T)$, which allows us to conclude that $(\widetilde v, p)$ is a solution of problem $(P)$.

\bigskip

There remains now to prove uniqueness. Let ($\widetilde{v}_{1}, p_{1}$) and $(\widetilde{v}_{2}, p_{2}$) be two solutions
of problem $(P)$. Since $\displaystyle \frac{ \partial \widetilde v_i}{\partial t} \in L^2 (0,T; {\cal V}_{0 div}')$, $i=1,2$, we may rewrite (\ref{NS-25}) as
\begin{eqnarray*}
&&\int_0^{T}
\bigr\langle \frac{\partial  \widetilde{v}_{i}}{\partial t} ,   \varphi\chi \bigl\rangle_{{\cal V}'_{0, div} , {\cal V}_{0, div}} \,  dt
 + \int_0^{T} a(\widetilde{v}_{i}+ G_0 \zeta ,\varphi\chi) \, dt
\\
&&
 +\int_{0}^{T}\int_{\Gamma_0} \ell \left(
|\widetilde{v}_{i} + \varphi \chi|-|\widetilde{v}_{i} |\right) dx' dt
\geq
\int_0^{T} (f,\varphi\chi )\, dt
- \int_0^{T}  \left(G_{0} \frac{\partial \zeta}{\partial t},\varphi\chi \right) \, dt 
\end{eqnarray*}
for any $\varphi \in {\cal V}_{0div}$ and $\chi \in {\cal D} (0, T)$. By density of ${\cal D}(0, T)\otimes{\cal V}_{0div}$  into $L^{2}(0 , T; {\cal V}_{0div})$ we may replace $\varphi \chi$ by $ (\widetilde{v}_{j}-\widetilde{v}_{i}) {\bf 1}_{[0,s]}$ with $i, j \in \{1, 2\}$, $i \not=j$ and $s \in [0, T]$.  By adding the two variational inequalities and using (\ref{coercif}) we get
\begin{eqnarray*}
 \frac{1}{2}\int_{0}^{s}\frac{d}{d t}\|\widetilde{v}_{1}-\widetilde{v}_{2}\|^{2}_{{\bf L}^{2}(\Omega)}dt
+ \alpha\int_{0}^{s}\|\widetilde{v}_{1}-\widetilde{v}_{2}\|^{2}_{{\bf H}^{1}(\Omega)} dt
 \leq 0.
\end{eqnarray*}
As $\widetilde{v}_{1}(0)=\widetilde{v}_{2}(0)=\widetilde{v}^{0}$, we obtain
\begin{eqnarray*}
 \|\widetilde{v}_{1}(s)-\widetilde{v}_{2}(s)\|^{2}_{ {\bf L}^{2}(\Omega)} \le 0 \quad \forall s \in [0, T].
\end{eqnarray*}
Then, with (\ref{NS-25}) we have
\begin{eqnarray*}
\bigl\langle \bigl(p_1 -p_2, div(\varphi)  \bigr), \chi \bigr\rangle_{{\cal D}'(0,T), {\cal D}(0,T)}  =0 \quad \forall \varphi \in {\bf H}^1_0 (\Omega), \ \forall \chi \in {\cal D}(0, T).
\end{eqnarray*}
Now let $\widetilde w \in L^2(\Omega)$ and
\begin{eqnarray*}
w= \widetilde w - \frac{1}{|\Omega|} \int_{\Omega} \widetilde w \, dx  \in L^2_0(\Omega).
\end{eqnarray*}
There exists $\varphi = P(w) \in {\bf H}^1_0 (\Omega)$ such that $div (\varphi) = w$ (see \cite{girault}) and thus
\begin{eqnarray*}
\begin{array}{ll}
\displaystyle  \bigl\langle (p_1 -p_2, \tilde w ), \chi \bigr\rangle_{{\cal D}'(0,T), {\cal D}(0,T)} \\
\displaystyle 
 =  \bigl\langle (p_1 -p_2,  w ), \chi \bigr\rangle_{{\cal D}'(0,T), {\cal D}(0,T)}  =0  \quad \forall \chi \in {\cal D}(0, T).
 \end{array}
\end{eqnarray*}
By density of ${\cal D}(0, T) \otimes L^2(\Omega)$ into $H^1_0 \bigl( 0, T; L^2(\Omega) \bigr)$ we get
\begin{eqnarray*}
\bigl\langle p_1 - p_2, \eta \bigr\rangle_{H^{-1}(0, T; L^2(\Omega)), H^1_0 (0, T; L^2(\Omega))} =0 \quad \forall \eta \in H^1_0 \bigl( 0,T; L^2(\Omega) \bigr)
\end{eqnarray*}
and thus $p_1 = p_2$. 

\end{proof}

\bigskip

Let us assume now that the following compatibility condition for the initial velocity is satisfied
\begin{eqnarray}\label{v0-1}
v^0 \in {\bf H}^2 (\Omega) , \quad div(v^0) =0 \  \mbox{\rm in} \ \Omega, \quad v^0 = g \ \mbox{\rm on} \ \partial \Omega
\end{eqnarray}
and
\begin{eqnarray} \label{v0}
\frac{\partial v^0}{\partial x_3} =0 \ \mbox{\rm on} \ \Gamma_0.
\end{eqnarray}
Then we may choose $G_0 =v^0$ and the initial condition (\ref{NS-25init}) becomes
\begin{eqnarray*}
\widetilde{v}(0, \cdot) = v^0 -  G_0 = \widetilde{v}^{0} = 0 \in  {\bf H}.
\end{eqnarray*}
Let us assume also that
\begin{eqnarray}\label{ell}
f \in  W^{1,2} \bigl(0,T;  {\bf L}^{ 2} (\Omega) \bigr), \quad  \ell \in W^{1,2} \bigl(0,T;  { L}^{ 2}(\Gamma_0; \br^+) \bigr) .
\end{eqnarray}

Then we can prove further regularity properties for the unique solution of problem $(P)$.

\begin{theorem} \label{tresca-regularity}
Let assumptions (\ref{zeta})-(\ref{v0-1})-(\ref{v0})-(\ref{ell}) hold.
Then the unique solution $(\widetilde v, p)$ of problem $(P)$ satisfies the following regularity properties
\begin{eqnarray*}
\frac{\partial \widetilde v}{\partial t} \in L^{\infty} \bigl( 0,T; {\bf L}^2 (\Omega) \bigr) \cap L^2 (0,T; {\cal V}_{0 div}), \quad p \in L^{\infty} \bigl(0,T; L^2_0 (\Omega) \bigr)
\end{eqnarray*}
and
\begin{eqnarray*}
 \frac{\partial^2 \widetilde v}{\partial t^2} \in  L^2 \bigl(0,T; \bigl( {\bf H}^1_{0 div}(\Omega) \bigr)' \bigr)
\end{eqnarray*}
with ${\bf H}^1_{0 div} (\Omega) = \bigl\{ \varphi \in {\bf H}^1 (\Omega): \ \varphi=0 \ \mbox{on $\partial \Omega$ and } div(\varphi) = 0 \ \mbox{in $\Omega$} \bigr\}$.
\end{theorem}

\begin{proof}
Let us adopt the same notations as in the previous proof. Recalling that the trace operator maps ${\bf H}^1 (\Omega)$ into ${\bf L}^4 (\partial \Omega)$ (see \cite{roubicek} for instance), 
 we infer from (\ref{ell}) that, for all $\varepsilon >0$ and for all $m \ge 1$, we have $\widetilde{v}_{\varepsilon m} \in W^{2,2}\bigl(0, T; {\rm Span} \{w_1, \dots, w_m \} \bigr)$ and (\ref{NS35}) holds for all $t \in [0,T]$.
We may differentiate all the terms of (\ref{NS35}) with respect to the time variable and we obtain
\begin{eqnarray}\label{eq38bis}
\begin{array}{ll}
\displaystyle  \left(\frac{\partial^{2} \widetilde{v}_{\varepsilon m}}{\partial t^{2}} , w_k \right)
 +
a\left( {\partial \widetilde{v}_{\varepsilon m} \over\partial t} , w_k \right)
+\int_{\Gamma_0} {\partial \ell\over\partial t}
\frac{\widetilde{v}_{\varepsilon m}  \cdot w_k
}{\sqrt{\varepsilon^{2}+|\widetilde{v}_{\varepsilon m} |^{2}}}\,dx'
%
\\
\displaystyle 
+ \int_{\Gamma_0}\ell
\left(
\frac{{\partial\widetilde{v}_{\varepsilon m}  \over\partial t} \cdot w_k}{\sqrt{\varepsilon^{2}+|\widetilde{v}_{\varepsilon m}|^{2}}}
- \frac{ \left( \widetilde{v}_{\varepsilon m} \cdot {\partial\widetilde{v}_{\varepsilon m}  \over\partial t} \right) ( \widetilde{v}_{\varepsilon m} \cdot w_k)}{ \bigl(\varepsilon^{2}+|\widetilde{v}_{\varepsilon m}|^{2}\bigr)^{3/2}}
 \right) \, dx'
%
%
%
\\
\displaystyle =
\left( {\partial f\over\partial t} , w_k \right)
-
 a \left( v^{0} {\partial \zeta\over \partial t} , w_k \right)
-
 \left( v^{0} {\partial^2 \zeta\over \partial t^2} , w_k \right)
 \quad \mbox{\rm  a.e. in $(0, T)$, }
 \end{array}
\end{eqnarray}
for all $k \in \{1, \dots, m\}$. 
Now we multiply (\ref{eq38bis}) by $g'_{\varepsilon k} (t)$ and we add from $k=1$ to $m$. We obtain
\begin{eqnarray*}
\int_{\Gamma_0}\ell
\left(
\frac{\left|{\partial\widetilde{v}_{\varepsilon m}  \over\partial t}\right|^2}{\sqrt{\varepsilon^{2}+|\widetilde{v}_{\varepsilon m}|^{2}}}
- \frac{ \left( \widetilde{v}_{\varepsilon m} \cdot {\partial\widetilde{v}_{\varepsilon m}  \over\partial t} \right)^2}{ \bigl(\varepsilon^{2}+|\widetilde{v}_{\varepsilon m}|^{2}\bigr)^{3/2}}
 \right) \, dx'
  \ge
 \int_{\Gamma_0}\ell\varepsilon^{2}
\frac{|{\partial\widetilde{v}_{\varepsilon m}  \over\partial t}|^{2}}{\left(\varepsilon^{2}
+|\widetilde{v}_{\varepsilon m} |^{2}\right)^{\frac{3}{2}}}\,dx'\geq 0
\end{eqnarray*}
and with (\ref{coercif}) we get
\begin{eqnarray}\label{eq39}
&& \frac{1}{2}\frac{d}{d t} \left\|
 {\partial \widetilde{v}_{\varepsilon m} \over\partial t}\right\|^{2}_{{\bf L}^{2}(\Omega)}
+ \alpha \left\| {\partial \widetilde{v}_{\varepsilon m} \over\partial t}\right\|^{2}_{{\bf H}^1(\Omega)}
\le
- \int_{\Gamma_0} {\partial \ell\over\partial t}
\frac{\widetilde{v}_{\varepsilon m} \cdot
{\partial\widetilde{v}_{\varepsilon m}  \over\partial t}
}{\sqrt{\varepsilon^{2}+|\widetilde{v}_{\varepsilon m} |^{2}}}\,dx'
\nonumber\\
&&
+ \left( {\partial f\over\partial t} ,  {\partial \widetilde{v}_{\varepsilon m}   \over\partial t}    \right)
-  
a\left( v^{0} {\partial \zeta\over \partial t} ,   {\partial  \widetilde{v}_{\varepsilon m}  \over\partial t}    \right)
-
\left( v^{0} {\partial^{2} \zeta \over\partial t^{2}} ,   {\partial  \widetilde{v}_{\varepsilon m} \over\partial t} \right)
 \  \mbox{\rm  a.e. in $(0, T)$.}
\end{eqnarray}
Let us  estimate now the right side of  (\ref{eq39}). We obtain
\begin{eqnarray*}
\begin{array}{ll}
\displaystyle \left|\int_{\Gamma_0}
\frac{\partial \ell}{\partial t}
\frac{\widetilde{v}_{\varepsilon m} \cdot \frac{\partial\widetilde{v}_{\varepsilon m}}{\partial t}}{\sqrt{\varepsilon^{2}
+|\widetilde{v}_{\varepsilon m}|^{2}}}\,dx'\right|
\\ 
\displaystyle \leq  c(\gamma_0)
\left\| \frac{\partial \ell}{\partial t} \right\|_{L^{2}(\Gamma_0)} \left\| \frac{\partial\widetilde{v}_{\varepsilon m}}{\partial t}\right\|_{{\bf H}^{1}(\Omega)}
\leq \frac{\alpha}{4}\left\| \frac{\partial\widetilde{v}_{\varepsilon m}}{\partial t}\right\|_{{\bf H}^{1}(\Omega)
}^{2}
+\frac{c(\gamma_0)^{2}}{\alpha} \left\| \frac{\partial \ell}{\partial t}\right\|_{L^{2}(\Gamma_0)}^{2}
\end{array}
\end{eqnarray*}
where  we recall that $c(\gamma_0)$ is the norm of the  trace operator from
${\bf H}^{1}(\Omega)$ to ${\bf L}^{2}(\Gamma_0)$,
\begin{eqnarray*}
\left|\left({\partial f\over\partial t},   {\partial\widetilde{v}_{\varepsilon m} \over\partial t}\right)\right|
\leq \left\|{\partial f\over\partial t}\right\|_{{\bf L}^{2}(\Omega)}
     \left\|{\partial\widetilde{v}_{\varepsilon m}\over\partial t}\right\|_{{\bf L}^{2}(\Omega)}
\leq
\left\|{\partial f\over\partial t}\right\|_{{\bf L}^{2}(\Omega)}^{2}
+\frac{1}{4}
\left\|{\partial\widetilde{v}_{\varepsilon m} \over\partial t}\right\|_{{\bf L}^{2}(\Omega)}^{2},
\end{eqnarray*}

\begin{eqnarray*}
\begin{array}{ll}
\displaystyle  \left| 
  a\left(v^{0} {\partial \zeta\over\partial t},  {\partial\widetilde{v}_{\varepsilon m} \over\partial t}\right)\right|
\leq 2 \mu \left|{\partial \zeta\over\partial t} \right|
\|v^{0}\|_{{\bf H}^{1}(\Omega)}
\left\|{\partial\widetilde{v}_{\varepsilon m} \over\partial t} \right\|_{{\bf H}^{1}(\Omega)}
\\
\displaystyle  
\leq \frac{\alpha}{4}\left\|{\partial\widetilde{v}_{\varepsilon m}\over\partial t} \right\|^{2}_{{\bf H}^{1}(\Omega)}
+\frac{ 4 \mu^{2}}{\alpha} \left|{\partial \zeta\over\partial t} \right|^{2}
\|v^{0}\|_{{\bf H}^{1}(\Omega)}^{2},
\end{array}
\end{eqnarray*}
\begin{eqnarray*}
\begin{array}{ll}
\displaystyle  \left|
\left(v^{0}  {\partial^{2} \zeta\over\partial t^{2}} , {\partial\widetilde{v}_{\varepsilon m}\over\partial t}\right)\right|
\leq
\left|{\partial^{2} \zeta\over\partial t^{2}} \right|
\|v^{0}\|_{{\bf L}^{2}(\Omega)}
\left\|{\partial\widetilde{v}_{\varepsilon m} \over\partial t}\right\|_{{\bf L}^{2}(\Omega)}
\\
\displaystyle 
\leq
 \left| {\partial^{2} \zeta\over\partial t^{2}} \right|^{2}
\|v^{0}\|_{{\bf L}^{2}(\Omega)}^{2}
+\frac{1}{4}
\left\|{\partial\widetilde{v}_{\varepsilon m} \over\partial t}\right\|_{{\bf L}^{2}(\Omega)}^{2}.
\end{array}
\end{eqnarray*}
Gathering all these estimates, we infer
\begin{eqnarray*}\label{NS47D}
 \frac{1}{2}\frac{d }{d t}
        \left\|{\partial\widetilde{v}_{\varepsilon m}\over\partial t}\right\|^{2}_{{\bf L}^{2}(\Omega)}
+ {\alpha\over 2}
         \left\| {\partial\widetilde{v}_{\varepsilon m}\over\partial t} \right\|^{2}_{{\bf H}^{1}(\Omega)}
 \leq   A_1 +  \frac{1}{2} \left\|{\partial\widetilde{v}_{\varepsilon m} \over\partial t}\right\|^{2}_{{\bf L}^{2}(\Omega)} \quad \mbox{\rm a.e. in $(0,T)$}
\end{eqnarray*}
with
\begin{eqnarray*}
A_1  = \frac{ c(\gamma_0)^{2}}{\alpha}\left\|{\partial \ell\over\partial t} \right\|_{L^{2}(\Gamma_0)}^{2}
+\left\|{\partial f\over\partial t} \right\|_{{\bf L}^{2}(\Omega)}^{2}
+ \frac{4 \mu^{2}}{\alpha} \left|{\partial \zeta\over\partial t}\right|^{2}
\|v^{0}\|_{{\bf H}^{1}(\Omega)}^{2}
 + \left|{\partial^2 \zeta\over\partial t^2}\right|^{2}
\|v^{0}\|_{{\bf L}^{2}(\Omega)}^{2}.
\end{eqnarray*}
Let $s \in [0, T]$. By integration we have
\begin{eqnarray} \label{NS47Dint}
\begin{array}{ll}
\displaystyle  \left\| {\partial\widetilde{v}_{\varepsilon m} \over\partial t}(s) \right\|^{2}_{{\bf L}^{2}(\Omega)}
+ \alpha
         \int_{0}^{s} \left\| {\partial\widetilde{v}_{\varepsilon m} \over\partial t} \right\|^{2}_{{\bf H}^{1}(\Omega)} \, dt \\ 
\displaystyle \leq    \left\| {\partial\widetilde{v}_{\varepsilon m} \over\partial t}(0) \right\|^{2}_{{\bf L}^{2}(\Omega)} + 2 \int_0^s A_1 \, dt  

   +  \int_{0}^{s}  \left\|{\partial\widetilde{v}_{\varepsilon m} \over\partial t} \right\|^{2}_{{\bf L}^{2}(\Omega)} \, dt.
\end{array}
 \end{eqnarray}
 
 
 Moreover, taking $t=0$ in (\ref{NS35})  we  obtain
\begin{eqnarray*}
&& \left( {\partial \widetilde{v}_{\varepsilon m} \over\partial t}(0),w_{k}\right)
=
-
a \bigl(\widetilde{v}_{\varepsilon m} (0) + v^0 \zeta(0),w_{k} \bigr)
-
\int_{\Gamma_0} \ell(0) \frac{\widetilde{v}_{\varepsilon m} (0) \cdot w_{k} }{\sqrt{\varepsilon^{2}
+ |\widetilde{v}_{\varepsilon m}(0)|^{2}}}\,dx'
\nonumber\\
&&+
\bigl(f(0), w_{k} \bigr)
 - \left( v^{0} \frac{\partial \zeta}{\partial t}(0) , w_{k} \right) \quad \forall k \in \{1, \dots, m\}.
\end{eqnarray*}
Reminding that $\zeta(0)=1$ and $G_0 = v^0$, we get $\widetilde{v}_{\varepsilon m} (0) = \widetilde{v}_{ m}^{ 0} =0 $ for all $m \ge 1$ and 
\begin{eqnarray*}
\int_{\Gamma_0} \ell(0) \frac{\widetilde{v}_{\varepsilon m} (0) \cdot w_{k} }
{\sqrt{\varepsilon^{2}
+ |\widetilde{v}_{\varepsilon m}(0)|^{2}}}\,dx' =0 .
\end{eqnarray*}
We multiply the previous equality by $g'_{\varepsilon k} (0)$ and we add from $k=1$ to $m$.  With Green's formula and (\ref{v0-1})-(\ref{v0}), we get
\begin{eqnarray*}\label{eq4.1}
&& \left\|
 {\partial \widetilde{v}_{\varepsilon m} \over\partial t}(0)\right\|^{2}_{{\bf L}^{2}(\Omega)}
= \int_{\Omega} 2 \mu \sum_{i,j =1}^3 \frac{\partial}{\partial x_j}\bigl( d_{ij}(v^0)  \bigr)
{\partial \widetilde{v}_{i \varepsilon m} \over\partial t}(0)
\, dx
\nonumber\\
&&  +
\left(f(0), {\partial \widetilde{v}_{\varepsilon m} \over\partial t}(0)\right)
 -   \left( v^0 \frac{\partial \zeta}{\partial t}(0),
 {\partial \widetilde{v}_{\varepsilon m} \over\partial t}(0) \right).
\end{eqnarray*}
It follows that
\begin{eqnarray*}
\left\|
{\partial \widetilde{v}_{\varepsilon m} \over\partial t}(0)
\right\|_{{\bf L}^{2}(\Omega)}
\leq A_0  =  2 \sqrt{3} \mu \| v^0 \|_{{\bf H}^{2}(\Omega)} +
\bigl\|f(0) \bigr\|_{{\bf L}^{2}(\Omega)}
 +\left|\frac{\partial \zeta}{\partial t}(0)\right| \| v^0 \|_{{\bf L}^{2}(\Omega)} .
\end{eqnarray*}
Finally, with (\ref{NS47Dint}) we get
\begin{eqnarray*} 
\begin{array}{ll}
\displaystyle  \left\| {\partial\widetilde{v}_{\varepsilon m} \over\partial t}(s) \right\|^{2}_{{\bf L}^{2}(\Omega)}
+ \alpha
         \int_{0}^{s} \left\| {\partial\widetilde{v}_{\varepsilon m} \over\partial t} \right\|^{2}_{{\bf H}^{1}(\Omega)} \, dt  
         \\ 
\displaystyle 
         \leq  \Bigl(  A_0^2 +  \frac{2 c(\gamma_0)^{2}}{\alpha}\left\|{\partial \ell\over\partial t} \right\|_{L^2(0,T; L^{2}(\Gamma_0))}^{2} 
         +2 \left\|{\partial f\over\partial t} \right\|_{L^2(0,T; {\bf L}^{2}(\Omega))}^{2}
         \\ 
\displaystyle   
+ \frac{8 \mu^{2}}{\alpha} T \left\|{\partial \zeta\over\partial t}\right\|^{2}_{{\cal C}([0,T])}
\|v^{0}\|_{{\bf H}^{1}(\Omega)}^{2} 
+ 2 T \left\|{\partial^2 \zeta\over\partial t^2}\right\|^{2}_{{\cal C}([0,T])}
\|v^{0}\|_{{\bf L}^{2}(\Omega)}^{2} 
\Bigr) \mbox{exp} (s) \quad \forall s \in [0,T].
\end{array}
 \end{eqnarray*}
Hence
 $\displaystyle \left( {\partial\widetilde{v}_{\varepsilon m} \over\partial t} \right)_{m\ge 1, \varepsilon>0}$ is  bounded in $L^{\infty} \bigl( 0,T; {\bf L}^{2}(\Omega) \bigr)$ and in $L^2 \bigl( 0,T; {\bf H}^{1}(\Omega) \bigr)$ uniformly with respect to $m$ and $\varepsilon$.
 
It follows that we can pass to the limit  as $m$ tends to $+ \infty$.
We obtain  
\begin{eqnarray} \label{derivee_v}
\frac{\partial \widetilde v_{\varepsilon m}}{\partial t} \rightharpoonup  \frac{\partial \widetilde v_{\varepsilon}}{\partial t} \quad \hbox{ weakly star in $ L^{\infty} \bigl( 0,T; {\bf L}^2 (\Omega) \bigr)$ and weakly in $ L^2 (0,T; {\cal V}_{0 div})$.}
\end{eqnarray}
Hence $\displaystyle  
\widetilde v_{\varepsilon} \in W^{1, \infty}  \bigl( 0,T; {\bf L}^2 (\Omega) \bigr) \cap W^{1,2} (0,T; {\cal V}_{0 div})$ and with a straightforward adaptation of Lemma \ref{lemma3} we get 
 \begin{eqnarray*} 
\begin{array}{ll}
\displaystyle 
\left| \bigl\langle ( { p}_{\varepsilon}  , \widetilde w ),  \chi \bigr\rangle_{{\cal D}'(0,T), {\cal D}(0,T) } \right|
\displaystyle  \le \left\| \frac{ \partial \widetilde{v}_{\varepsilon}}{\partial t} \right\|_{L^{\infty} (0,T; {\bf L}^{2} (\Omega))} \| P(w) \chi \|_{L^{1} (0,T; {\bf L}^{2}(\Omega))} \\
\displaystyle + 2 \mu  \bigl( \| \widetilde{v}_{\varepsilon} \|_{L^{\infty} (0,T;{\bf H}^1(\Omega) )} +  \|v^0\|_{{\bf H}^1(\Omega)} \| \zeta \|_{C^0([0,T]} \bigr) \| P(w) \chi \|_{L^{1} (0,T; {\bf H}^{1}(\Omega))} \\
\displaystyle  + \left( \|f\|_{L^{\infty} (0,T;{\bf L}^{2}(\Omega))} + \|v^0\|_{{\bf L}^2(\Omega)} \left\| \frac{\partial \zeta}{\partial t} \right\|_{C^0([0,T]} \right) \| P(w) \chi \|_{L^{1} (0,T; {\bf L}^{2}(\Omega))}
\\
\displaystyle \forall \chi \in  {\cal D} (0,T).
\end{array}
\end{eqnarray*}
Hence $p_{\varepsilon} \in L^{\infty}  \bigl(0,T; L^2_0 (\Omega) \bigr)$
and the sequence $\displaystyle \left( \frac{\partial \widetilde v_{\varepsilon}}{\partial t} , p_{\varepsilon} \right)_{\varepsilon >0}$ is uniformly bounded in $ \Bigl(  L^{\infty} \bigl( 0,T; {\bf L}^2 (\Omega) \bigr) \cap L^2 (0,T; {\cal V}_{0 div}) \Bigr) \times L^{\infty}  \bigl(0,T; L^2_0 (\Omega) \bigr)$.
Moreover, for all $\varphi \in {\cal V}_{0 div}$ and $m \ge 1$ we may define $\varphi_{m}$ as the othogonal projection of $\varphi$ with respect to the  inner product of
${\bf H}^{1}(\Omega)$
 on   ${\rm Span} \bigl\{w_{1},\ldots,w_{m} \bigr\}$. We multiply  (\ref{eq38bis})
 by $\chi$ and we integrate on $(0,T)$. We obtain 
 \begin{eqnarray} \label{eq38ter}
&& - \int_0^T \left(\frac{\partial \widetilde{v}_{\varepsilon m}}{\partial t} , \varphi_m \right) \chi' \, dt
 +
\int_0^T a\left( {\partial \widetilde{v}_{\varepsilon m} \over\partial t} , \varphi_m \right) \chi \, dt 
\nonumber\\
&& 
+ \int_0^T \int_{\Gamma_0} {\partial \ell\over\partial t}
\frac{\widetilde{v}_{\varepsilon m}  \cdot \varphi_m
}{\sqrt{\varepsilon^{2}+|\widetilde{v}_{\varepsilon m} |^{2}}} \chi \,dx' dt
\nonumber\\
&& 
+ \int_0^T \int_{\Gamma_0}\ell
\left(
\frac{{\partial\widetilde{v}_{\varepsilon m}  \over\partial t} \cdot \varphi_m}{\sqrt{\varepsilon^{2}+|\widetilde{v}_{\varepsilon m}|^{2}}}
- \frac{ \left( \widetilde{v}_{\varepsilon m} \cdot {\partial\widetilde{v}_{\varepsilon m}  \over\partial t} \right) ( \widetilde{v}_{\varepsilon m} \cdot \varphi_m)}{ \bigl(\varepsilon^{2}+|\widetilde{v}_{\varepsilon m}|^{2}\bigr)^{3/2}}
 \right) \chi \, dx' dt
\\
&& =
\int_0^T \left( {\partial f\over\partial t} , \varphi_m \right) \chi \, dt 
-
\int_0^T a \left( v^{0} {\partial \zeta\over \partial t} , \varphi_m \right) \chi \, dt 
\nonumber\\
&& 
-
\int_0^T \left( v^{0} {\partial^2 \zeta\over \partial t^2} , \varphi_m \right) \chi \,  dt 
 \quad \mbox{\rm  for all $m \ge 1$. } \nonumber
\end{eqnarray}
Since the trace operator is a linear continuous mapping from ${\bf H}^1(\Omega)$ into ${\bf L}^4(\partial \Omega)$ we infer from (\ref{derivee_v}) that
 \begin{eqnarray*} 
\frac{\partial \widetilde v_{\varepsilon m}}{\partial t} \rightharpoonup  \frac{\partial \widetilde v_{\varepsilon}}{\partial t} \quad \hbox{  weakly in $ L^2 (0,T; {\bf L}^4 \bigl(\Gamma_0) \bigr)$}
\end{eqnarray*}
under the convention that we identify $\displaystyle \frac{\partial \widetilde v_{\varepsilon m}}{\partial t}$ (resp. $\displaystyle \frac{\partial \widetilde v_{\varepsilon }}{\partial t}$) with its trace on $\Gamma_0$.  Moreover, with (\ref{conver6}) we have 
\begin{eqnarray*}
\begin{array}{ll}
\displaystyle \ell
\left(
\frac{ \varphi_m}{\sqrt{\varepsilon^{2}+|\widetilde{v}_{\varepsilon m}|^{2}}}
- \frac{  \widetilde{v}_{\varepsilon m} ( \widetilde{v}_{\varepsilon m} \cdot \varphi_m)}{ \bigl(\varepsilon^{2}+|\widetilde{v}_{\varepsilon m}|^{2}\bigr)^{3/2}}
 \right) \\
 \displaystyle 
 \to 
 \ell
\left(
\frac{ \varphi }{\sqrt{\varepsilon^{2}+|\widetilde{v}_{\varepsilon}|^{2}}}
- \frac{  \widetilde{v}_{\varepsilon } ( \widetilde{v}_{\varepsilon } \cdot \varphi)}{ \bigl(\varepsilon^{2}+|\widetilde{v}_{\varepsilon }|^{2}\bigr)^{3/2}}
 \right) \quad \hbox{ strongly in $L^2 (0,T; {\bf L}^{4/3} \bigl(\Gamma_0) \bigr)$.}
 \end{array}
 \end{eqnarray*}
Hence we can pass to the limit in all the terms of (\ref{eq38ter}) and we get
\begin{eqnarray*} 
&& - \int_0^T \left( {\partial \widetilde{v}_{\varepsilon} \over\partial t } , \varphi  \right) \chi' \, dt 
 +
\int_0^T a\left( {\partial \widetilde{v}_{\varepsilon } \over\partial t} , \varphi \right) \chi \, dt
+ \int_0^T \int_{\Gamma_0} {\partial \ell\over\partial t}
\frac{\widetilde{v}_{\varepsilon }  \cdot \varphi
}{\sqrt{\varepsilon^{2}+|\widetilde{v}_{\varepsilon } |^{2}}}  \chi \,dx' dt
%
\nonumber\\
&& 
+ \int_0^T \int_{\Gamma_0}\ell
\left(
\frac{{\partial\widetilde{v}_{\varepsilon }  \over\partial t} \cdot \varphi }{\sqrt{\varepsilon^{2}+|\widetilde{v}_{\varepsilon }|^{2}}}
- \frac{ \left( \widetilde{v}_{\varepsilon } \cdot {\partial\widetilde{v}_{\varepsilon }  \over\partial t} \right) ( \widetilde{v}_{\varepsilon } \cdot \varphi )}{ \bigl(\varepsilon^{2}+|\widetilde{v}_{\varepsilon }|^{2}\bigr)^{3/2}}
 \right) \chi \, dx' dt
\nonumber\\
&& =
\int_0^T \left( {\partial f\over\partial t} , \varphi \right) \chi \, dt 
-
\int_0^T a \left( v^{0} {\partial \zeta\over \partial t} , \varphi \right) \chi \, dt
-
\int_0^T \left( v^{0} {\partial^2 \zeta\over \partial t^2} , \varphi \right) \chi \, dt
\end{eqnarray*}
for all $\varphi \in {\cal V}_{0 div}$ and for all $\chi \in {\cal D}(0,T)$. Let us choose now  $\varphi \in {\bf H}^1_{0 div} (\Omega)$. We obtain
\begin{eqnarray*} 
&& \left\langle  \frac{d}{dt} \left( {\partial \widetilde{v}_{\varepsilon} \over\partial t } , \varphi  \right),  \chi \right\rangle_{{\cal D}' (0,T), {\cal D}(0,T)} 
 +
\int_0^T a\left( {\partial \widetilde{v}_{\varepsilon } \over\partial t} , \varphi \right) \chi \, dt
\nonumber\\
&& =
\int_0^T \left( {\partial f\over\partial t} , \varphi \right) \chi \, dt 
-
\int_0^T a \left( v^{0} {\partial \zeta\over \partial t} , \varphi \right) \chi \, dt
\nonumber\\
&& 
-
\int_0^T \left( v^{0} {\partial^2 \zeta\over \partial t^2} , \varphi \right) \chi \, dt
 \quad  \forall \chi \in {\cal D}(0,T)
\end{eqnarray*}
and we infer from the previous estimates that $\displaystyle  \frac{\partial^2 {\widetilde v}_{\varepsilon} }{\partial t^2} $ belongs to $  L^2 \bigl(0,T; \bigl( {\bf H}^1_{0 div}(\Omega) \bigr)' \bigr)$ and is uniformly bounded with respect to $\varepsilon$ in $  L^2 \bigl(0,T; \bigl( {\bf H}^1_{0 div}(\Omega) \bigr)' \bigr)$.

 Finally, by using the techniques of Theorem \ref{tresca-existence}, we may conclude that the unique solution $(\widetilde v, p)$ of problem $(P)$ satisfies 
\begin{eqnarray*}
\frac{\partial \widetilde v}{\partial t} \in L^{\infty} \bigl( 0,T; {\bf L}^2 (\Omega) \bigr) \cap L^2 (0,T; {\cal V}_{0 div}), \quad p \in L^{\infty}  \bigl(0,T; L^2_0 (\Omega) \bigr)
\end{eqnarray*}
and
\begin{eqnarray*}
\quad \frac{\partial^2 \widetilde v}{\partial t^2} \in  L^2 \bigl(0,T; \bigl( {\bf H}^1_{0 div}(\Omega) \bigr)' \bigr).
\end{eqnarray*}
 
\end{proof}

As a corollary we obtain

\begin{proposition} \label{proposition1}
Let us assume that  (\ref{zeta})-(\ref{v0-1})-(\ref{v0})-(\ref{ell})  hold.
 Then the stress tensor  $\sigma = - p {\rm Id} + 2 \mu D(\widetilde v + v_0 \zeta)$ belongs to  $ L^{\infty} \bigl( 0,T; \bigl( L^2(\Omega) \bigr)^{3 \times 3} \bigr)$ and $div(\sigma)$ belongs to $ L^{\infty} \bigl( 0, T; {\bf L}^{2} (\Omega) \bigr)$.
\end{proposition}

\begin{proof}
The first part of the result is an immediate consequence of Theorem \ref{tresca-regularity}. Now let us choose $\varphi \in \bigl({\cal D}(\Omega)\bigr)^3$ and $\chi \in {\cal D}(0, T)$. With (\ref{NS-25}) we have
\begin{eqnarray*}
&&\left\langle {d \over d t} \left( \widetilde{v}, \varphi \right) ,\pm \chi\right\rangle_{{\cal D}'(0,T), {\cal D}(0,T)} 
+ \int_0^T a(\widetilde{v}+ v^0 \zeta, \pm \varphi \chi) \, dt 
 \nonumber\\
&&
- \bigl\langle \bigl(p , div(\varphi)  \bigr), \pm \chi \bigr\rangle_{{\cal D}'(0,T), {\cal D}(0,T)} 
 \nonumber\\
&&
  \geq
\bigl\langle (f,\varphi ), \pm \chi \bigr\rangle_{{\cal D}'(0,T), {\cal D}(0,T)} 
 - \left\langle \left(v^{0} \frac{\partial \zeta}{\partial t},\varphi \right),
 \pm \chi \right\rangle_{{\cal D}'(0,T), {\cal D}(0,T)}  .
\end{eqnarray*}
It follows that
\begin{eqnarray*}\label{sigma}
\begin{array}{ll}
\displaystyle 
\int_0^T \int_{\Omega} {\partial  \widetilde{v} \over\partial t} \varphi \chi \, dx dt 
+ \int_0^{T} \int_{\Omega} \sum_{i, j =1}^3  \sigma_{ij} \frac{\partial \varphi_i}{\partial x_j} \chi \, dx dt
\\
\displaystyle  =
\int_0^ T \int_{\Omega} f \varphi \chi \, dx dt 
 - \int_0^ T \int_{\Omega} v^{0} \frac{\partial \zeta}{\partial t} \varphi \chi \, dx dt 
\end{array}
\end{eqnarray*}
and thus
\begin{eqnarray*}
&& \left|  \int_0^{T} \int_{\Omega} \sum_{i, j = 1}^3 \sigma_{ij} \frac{\partial \varphi_i}{\partial x_j} \chi \, dx dt  \right| 
 \le
\left(
\left\| \frac{\partial \widetilde{v} }{\partial t} \right\|_{L^{\infty} (0, T; {\bf L}^2(\Omega))}
+ \|f\|_{L^{\infty} (0, T; {\bf L}^2(\Omega))} \right. \\
&& \left.
+ \left\| \frac{\partial \zeta }{\partial t} \right\|_{C^{0} ([0, T]} \|v^0\|_{{\bf L}^2(\Omega)} \right) \| \varphi \chi \|_{L^1 (0, T; {\bf L}^2(\Omega))}
.
\end{eqnarray*}
Hence $div(\sigma) \in \bigl(  L^1\bigl( 0, T; {\bf L}^{2} (\Omega) \bigr) \bigr)'= L^{\infty}\bigl( 0, T; {\bf L}^{2} (\Omega) \bigr)$.

\end{proof}


\renewcommand{\theequation}{4.\arabic{equation}}
 \setcounter{equation}{0}

\section{The generalized Coulomb's friction case} \label{coulomb}

Let us assume now that ${\cal F}$ can be decomposed as 
\begin{eqnarray}\label{coulomb.1}
 {\cal F}(x', t, \sigma) ={\cal F}^0(x',t) +  {\cal F}^{\sigma} (x',t)   \int_{0}^{t} S(t-s) \bigl| {\cal R}(\sigma^3(\cdot, s))(x') \bigr| \, ds 
\end{eqnarray}
 for almost every $x' \in \Gamma_0$ and for all $t \in [0,T]$, where ${\cal R}$ is a regularization operator
 which will be described below, $\sigma^3$ is the vector $(\sigma_{3j})_{1\le j \le 3}$ and 
\begin{eqnarray} \label{coulomb.2}
{\cal F}^0 \in W^{1,2} \bigl( 0,T; L^2 (\Gamma_0; \br^+) \bigr), \quad  {\cal F}^{\sigma} \in W^{1, p} \bigl( 0,T; L^{2} (\Gamma_0; \br^+) \bigr) \ \hbox{ with $p>2$,}
\end{eqnarray}
and
\begin{eqnarray} \label{coulomb.2bis}
 S \in C^1(\br^+; \br^+).
\end{eqnarray}
Then, possibly modifying ${\cal F}^0$ and ${\cal F}^{\sigma}$ on a negligible subset of $[0,T]$, we have  
\begin{eqnarray*}
{\cal F}^0 \in C^0 \bigl( [0,T] ; L^2 (\Gamma_0; \br^+) \bigr), \  {\cal F}^{\sigma} \in C^0 \bigl( [0,T] ; L^{2} (\Gamma_0; \br^+) \bigr) .
\end{eqnarray*}
Similarly, recalling that $f \in W^{1,2} \bigl(0,T; {\bf L}^2(\Omega) \bigr)$, possibly modifying $f$ on a negligible subset of $[0,T]$, we have also $f \in C^0 \bigl( [0,T]; {\bf L}^2(\Omega) \bigr)$.

\bigskip

We will prove an existence result for problem $(P)$ by using a successive approximation technique. Indeed, for any given Tresca's friction threshold $ \ell_k \in W^{1,2} \bigl( 0,T; L^2 (\Gamma_0; \br^+) \bigr)$, we consider 

\noindent {\bf Problem $(P_k)$}
Find
$$\widetilde{v}_k\in L^{2} \bigl(0,T;{\cal V}_{0div} \bigr)\cap L^{\infty}\bigl(0,T;{\bf L}^{2}(\Omega) \bigr),\  \,
p_k\in H^{-1}\bigl(0,T;L^{2}_{0}(\Omega) \bigr)$$
such that, for all $ \varphi\in {\cal V}_{0}$ and for all $\chi\in {\cal D}(0,T)$, we have
\begin{equation}\label{NS-25-k}
\begin{array}{ll}
\displaystyle
\left\langle \frac{d}{dt} \left( \widetilde{v}_k, \varphi \right) , \chi\right\rangle_{{\cal D}'(0,T), {\cal D}(0,T)}  
 - \bigl\langle \bigl(p_k,div(\varphi)  \bigr), \chi \bigr\rangle_{{\cal D}'(0,T), {\cal D}(0,T)} \\
\displaystyle + \int_0^T a(\widetilde{v}_k,\varphi \chi ) \, dt 
+  \int_0^T \int_{\Gamma_0} \ell_k \bigl( \bigl| \widetilde{v}_k+ \varphi \chi  \bigr| - \bigl| \widetilde{v}_k \bigr|  \bigr) \, dx' dt \\
\displaystyle  \geq
\bigl\langle (f,\varphi ), \chi \bigr\rangle_{{\cal D}'(0,T), {\cal D}(0,T)} -  \int_0^T a(v^{0}\zeta ,\varphi \chi) \, dt  
\displaystyle - \left\langle \left(v^{0}\frac{\partial \zeta}{\partial t},\varphi \right),
\chi \right\rangle_{{\cal D}'(0,T ), {\cal D}(0,T)} 
\end{array}
\end{equation}
and 
\begin{eqnarray}\label{NS-25init-k}
\widetilde{v}_k(0, \cdot) = \widetilde{v}^{0} =0.
\end{eqnarray}

With the results of Section \ref{tresca} we know that problem $(P_k)$ admits an unique solution and  $\sigma_k \eqldef - p_k {\rm Id} + 2 \mu D(\widetilde v_k + v_0 \zeta)$ belongs to  $ L^{\infty} \bigl( 0,T; \bigl( L^2(\Omega) \bigr)^{3 \times 3} \bigr)$ with $div(\sigma_k) \in L^{\infty} \bigl( 0, T; {\bf L}^{2} (\Omega) \bigr)$. 
For all $i \in \{1, 2, 3 \}$ we denote by $\sigma_k^i$ the vector $(\sigma_{ ij k})_{1 \le j \le 3}$ and we introduce the following functional space 
\begin{eqnarray*}
E(\Omega) = \bigl\{ \widetilde \sigma \in {\bf L}^2(\Omega); \  div( \widetilde \sigma) \in L^2(\Omega) \}
\end{eqnarray*}
endowed with the norm $\| \cdot \|_{E(\Omega)}$ given by
\begin{eqnarray*}
\| \widetilde \sigma \|_{E(\Omega)} = \bigl( \| \widetilde \sigma \|^2_{{\bf L}^2(\Omega)} + \bigl\| div(\widetilde \sigma) \bigr\|^2_{L^2(\Omega)} \bigr)^{1/2} \quad \forall \widetilde \sigma \in E(\Omega).
\end{eqnarray*}
With the previous results we know that $\sigma_k^i \in L^{\infty} \bigl(0,T; E(\Omega) \bigr)$ for all $i \in \{1, 2, 3 \}$ and for all $k \ge 0$. Let us recall that there exists a trace operator $\gamma_n : E(\Omega) \to H^{-1/2} (\partial \Omega)$ such that, for all $\widetilde \sigma \in E(\Omega)$,  the following Green's formula holds
\begin{eqnarray*}
\int_{\Omega} div (\widetilde \sigma) \psi \, dx + \int_{\Omega} \widetilde \sigma \cdot \nabla \psi \, dx = \bigl\langle \gamma_n(\widetilde \sigma), \psi \bigr\rangle_{H^{-1/2} (\partial \Omega), H^{1/2}(\partial \Omega) } \quad \forall \psi \in H^1(\Omega)
\end{eqnarray*}
and $\gamma_n (\widetilde \sigma)$ is called the normal component of $\widetilde \sigma$ on $\partial \Omega$ (see \cite{girault} for instance). Then, following \cite{demkowicz} we introduce a regularization operator ${\cal R} : E ( \Omega) \to C^0(\Gamma_0)$ given by
\begin{eqnarray} \label{calR}
\begin{array}{ll} 
\displaystyle {\cal R} (\widetilde \sigma) (x')  = \bigl\langle \gamma_n ( \widetilde \sigma ) , f_{x'} \bigr\rangle_{H^{-1/2} (\partial \Omega), H^{1/2}(\partial \Omega) } \\
\displaystyle = \int_{\Omega} div (\widetilde \sigma) f_{x'} \, dx + \int_{\Omega} \widetilde \sigma \cdot \nabla f_{x'} \, dx \quad \forall x' \in \Gamma_0, \ \forall \widetilde \sigma \in E(\Omega)
\end{array}
\end{eqnarray} 
where $f$ is a function belonging to $ C^{\infty}_0 (\Gamma_0 \times \br^3; \br^+)$ and $f_{x'}: \Omega \to \br$ is defined by $f_{x'}(x)  = f(x', x' - x)$ for all $x \in \Omega$ and for all $x' \in \Gamma_0$.
%
Since $\gamma_n (\widetilde \sigma) = \widetilde \sigma \cdot n$ on $\partial \Omega$ for all $\widetilde \sigma \in \bigl( {\cal D} ({\overline \Omega}) \bigr)^3$ and $n=(0,0, -1)$ on $\Gamma_0$, we obtain that 
\begin{eqnarray*}
\sigma_n = \sum_{i, j =1}^3 \sigma_{ij} n_i n_j =  - \gamma_n (\sigma^3) \quad \hbox{ on $\Gamma_0$}
\end{eqnarray*}
for any $\sigma \in \bigl( {\cal D} ({\overline \Omega}) \bigr)^{3 \times 3}$. Then  we let
\begin{eqnarray*}\label{coulomb.3}
\ell_{k+1} (x', t) =  {\cal F}(x', t, \sigma_k) ={\cal F}^0(x',t) +  {\cal F}^{\sigma} (x',t)   \int_{0}^{t} S(t-s) \bigl| {\cal R}(\sigma_{k}^3(\cdot, s)) (x') \bigr| \, ds 
\end{eqnarray*}
 for almost every $x' \in \Gamma_0$ and for all $t \in [0,T]$, thus $\ell_{k+1} \in W^{1,2} \bigl( 0,T; L^2 (\Gamma_0; \br^+) \bigr)$. Starting from a given $\ell_0 \in W^{1,2} \bigl( 0,T; L^2 (\Gamma_0, \br^+) \bigr)$, we construct a sequence $(\widetilde v_k, p_k, \sigma_k)_{k \ge 0}$ such that 
\begin{eqnarray*}
 \widetilde v_k  \in W^{1,\infty} \bigl( 0,T; {\bf L}^2 (\Omega) \bigr) \cap W^{1,2} (0,T; {\cal V}_{0 div}), 
 \end{eqnarray*}
 \begin{eqnarray*}
  p_k \in L^{\infty} \bigl(0,T; L^2_0 (\Omega) \bigr),
  \end{eqnarray*}
  \begin{eqnarray*}
\sigma_k \in L^{\infty} \bigl( 0,T; \bigl( L^2(\Omega) \bigr)^{3 \times 3} \bigr), \quad div(\sigma_k) \in L^{\infty} \bigl( 0, T; {\bf L}^{2} (\Omega) \bigr)
\end{eqnarray*}
and we expect that $(\widetilde v_k, p_k)_{k \ge 0}$ converges towards a solution of problem $(P)$. 

\bigskip

More precisely, let us assume that problem $(P)$ admits a solution $(\bar v, \bar p)$ on some time interval $[0, \tau]$, with $0 \le \tau <T$. 
Let $\bar \sigma = - \bar p {\rm Id} + 2 \mu D(\bar v + v^0 \zeta)$ and assume that $\bar \sigma \in L^{\infty} \bigl( 0,T; \bigl( L^2(\Omega) \bigr)^{3 \times 3} \bigr)$, $ div(\bar \sigma) \in L^{\infty} \bigl( 0, T; {\bf L}^{2} (\Omega) \bigr)$.
We define $\ell_0$ by
\begin{eqnarray*} \label{coulomb.4}
\ell_0 (x',t) =
{\cal F}^0(x',t) +  {\cal F}^{\sigma} (x',t)   \int_{0}^{t} S(t-s) \bigl| {\cal R}( \bar \sigma^{3}(\cdot, s)) (x') \bigr| \, ds 
\end{eqnarray*}
for almost every $x' \in \Gamma_0$ and for all $t \in [0,\tau]$ and 
\begin{eqnarray*} \label{coulomb.5}
\ell_0 (x',t) = 
{\cal F}^0(x',t)  +  {\cal F}^{\sigma} (x',\tau)   \int_{0}^{\tau} S(\tau -s) \bigl| {\cal R}(\bar \sigma^{3}(\cdot, s)) (x') \bigr| \, ds 
\end{eqnarray*}
for almost every $x' \in \Gamma_0$ and for all $t \in [\tau, T]$. 

\bigskip

Let $(\widetilde v_0, p_0)$ be the unique solution of problem $(P_0)$ 
on $[0,T]$. Then $(\widetilde v_0, p_0)$ satisfies also problem $(P)$ with $  {\cal F} ( \cdot, \cdot, \bar \sigma) =  {\ell_0}_{|[0,\tau]} $ on $[0,\tau]$ and, by uniqueness of the solution (see Theorem \ref{tresca-existence}), we infer that $\widetilde v_0 = \bar v$ and $p_0 = \bar p$ on $(0,\tau)$.

Next let $(\ell_k)_{k \ge 0}$ be given by the iteration procedure described previously i.e.
\begin{eqnarray*}
\ell_{k+1} (x',t)  =  {\cal F}(x', t, \sigma_{k}) ={\cal F}^0(x',t) +  {\cal F}^{\sigma} (x',t)   \int_{0}^{t} S(t-s) \bigl| {\cal R}( \sigma_{k}^3(\cdot, s)) (x') \bigr| \, ds 
\end{eqnarray*}
for almost every $x' \in \Gamma_0$ and for all $t \in [0,T]$, with $\sigma_k = - p_k {\rm Id} + 2 \mu D(\widetilde v_k + v^0 \zeta)$  for all $k \ge 0$. Let $(\widetilde v_1, p_1)$ be the unique solution of problem $(P_1)$ 
on $[0,T]$. We observe that $ \sigma_0 = \bar \sigma $ on $(0,\tau)$, hence
$\ell_1=\ell_0$ on $(0,\tau)$ and we obtain $\widetilde v_1 = \widetilde v_0 = \bar v$ and $p_1 = p_0= \bar p$ on $(0,\tau)$. By an immediate induction we get
\begin{eqnarray*}
\ell_k=\ell_0, \quad \widetilde v_k = \bar v, \quad  p_k = \bar p \quad \mbox{on $(0,\tau)$, for all $k \ge 0$.}
\end{eqnarray*}

\bigskip

Now let $\tau' \in (\tau, T]$. By using the results of Theorem \ref{tresca-existence}, Theorem \ref{tresca-regularity} and Proposition \ref{proposition1}
we obtain that there exists a positive real number $C_{ \rm data} = C_{\rm data}(\mu, \zeta, f, v^0)$, depending only on the data $\mu$, $\zeta$, $f$ and $v^0$, such that $\widetilde v_k $, $\displaystyle \frac{\partial^2 \widetilde v_k}{\partial t^2}$, $p_k$, $\sigma_k$ and $div(\sigma_k)$ are bounded by $C_{\rm data} \bigl( 1 + \| \ell_k\|_{W^{1,2}(0, \tau'; L^2(\Gamma_0))} \bigr) $ in $W^{1,\infty} \bigl(0, \tau'; {\bf L}^2(\Omega) \bigr) \cap W^{1,2} (0, \tau'; {\cal  V}_{0 div}) $, $L^2 \bigl(0, \tau', \bigl({\bf H}^1_{0 div} (\Omega) \bigr)' \bigr)$, $L^{\infty} \bigl(0, \tau'; L^2_0(\Omega) \bigr)$, $L^{\infty} \bigl(0, \tau';  (L^2(\Omega)^{3\times 3} \bigr)$ and $L^{\infty} \bigl(0, \tau'; {\bf L}^2(\Omega) \bigr)$ respectively.
It follows that 
%
%
\begin{eqnarray*} 
\| \sigma_k^3 \|_{L^{\infty}(0, \tau'; E( \Omega))} \le \sqrt{2} C_{\rm data} \bigl( 1 + \| \ell_k\|_{W^{1,2}(0, \tau'; L^2(\Gamma_0))} \bigr) \quad \forall k \ge 0.
\end{eqnarray*}


Moreover, let $\widetilde \ell_k = \ell_k - {\cal F}^0$ for all $k \ge 0$. We have:

\begin{proposition} \label{estimation}
Under the previous assumptions, there exist a constant $C_{\tau}$, depending only on $\tau$, and a constant $C'_{\rm data} = C'_{\rm data}(\mu, \zeta, f, v^0)$, depending only on the data $\mu$, $\zeta$, $f$ and $v^0$,  such that
\begin{eqnarray*} \label{point-fixe2}
\| \widetilde \ell_{k+1} \|_{W^{1,2}(0, \tau'; L^2(\Gamma_0))} \le C_{\tau}
 +  C'_{\rm data}\bigl(  \sqrt{\tau' - \tau } + (\tau' - \tau)^{\frac{p-2}{2p}} \bigr)
 \| \widetilde \ell_{k} \|_{W^{1,2}(0, \tau'; L^2(\Gamma_0))} 
\end{eqnarray*}
for all $k \ge 0$.
\end{proposition}

\begin{proof}
For all $k \ge 0$ we have
\begin{eqnarray*}
\widetilde \ell_{k+1} (x',t) =  {\cal F}^{\sigma} (x',t)   \int_{0}^{t} S(t-s) \bigl| {\cal R}( \sigma_{k}^3(\cdot, s)) (x') \bigr| \, ds 
\end{eqnarray*}
for almost every $x' \in \Gamma_0$ and for all  $t \in [0,T]$ and 
\begin{eqnarray*}
\begin{array}{ll}
\displaystyle \frac{ \partial \widetilde \ell_{k+1}}{\partial t} (x',t) = \frac{\partial {\cal F}^{\sigma}}{\partial t} (x',t)   \int_{0}^{t} S(t-s) \bigl|  {\cal R}( \sigma_{k}^3(\cdot, s)) (x')\bigr| \, ds \\
\displaystyle + {\cal F}^{\sigma} (x',t)   S(0) \bigl| {\cal R}( \sigma_{k}^3(\cdot, t)) (x') \bigr| 
+ {\cal F}^{\sigma} (x',t)   \int_{0}^{t} S' (t-s) \bigl| {\cal R}( \sigma_{k}^3(\cdot, s)) (x') \bigr| \, ds  
\end{array}
\end{eqnarray*}
for almost every $x' \in \Gamma_0$ and for almost every  $t \in [0,T]$.

But ${\cal F}^{\sigma} \in W^{1, p} \bigl( 0,T; L^{2} (\Gamma_0; \br^+) \bigr)$, $S \in C^1 (\br^+ ; \br^+)$ and 
${\cal R} \in {\cal L}_c \bigl( E( \Omega) ; C^0(\Gamma_0 ) \bigr)$. So there exists three positive real numbers $ C_{{\cal F}^{\sigma}}$, $C_{S}$ and $C_{\cal R}$ such that
\begin{eqnarray*}
\bigl\| {\cal F}^{\sigma}  \bigr\|_{L^{\infty} (0,T; L^{2}(\Gamma_0))} \le C_{{\cal F}^{\sigma}}, \quad 
\left\| \frac{\partial {\cal F}^{\sigma}}{\partial t}  \right\|_{L^{p} (0,T; L^{2}(\Gamma_0))} \le C_{{\cal F}^{\sigma}},
\end{eqnarray*}
\begin{eqnarray*}
\bigl| S(t') \bigr| \le C_S, \quad \bigl| S' (t') \bigr| \le C_S \quad \forall t' \in [0,T],
\end{eqnarray*}
and 
\begin{eqnarray*}
\bigl\| {\cal R} (\widetilde \sigma) \bigr\|_{L^{\infty} (\Gamma_0)} \le C_{{\cal R}} \| \widetilde \sigma \|_{E ( \Omega)} \quad \forall \widetilde \sigma  \in E ( \Omega) .
\end{eqnarray*}
It follows that 
\begin{eqnarray*}
\begin{array}{ll}
\displaystyle \bigl\| \widetilde \ell_{k+1} \bigr\|^2_{L^2(0, \tau'; L^2(\Gamma_0) ) }  = \int_0^{\tau'} \left\|  {\cal F}^{\sigma} (x',t)   \int_{0}^{t} S(t-s) \bigl| {\cal R}( \sigma_{k}^3(\cdot, s)) (x') \bigr| \, ds \right\|^2_{L^2(\Gamma_0)} \, dt \\
\displaystyle =  \| \widetilde \ell_0 \|_{L^2(0, \tau; L^2(\Gamma_0))}^2 + \int_{\tau}^{\tau'} \left\|  {\cal F}^{\sigma} (x',t)   \int_{0}^{t} S(t-s) \bigl|  {\cal R}(\sigma_{k}^3(\cdot, s)) (x') \bigr| \, ds \right\|^2_{L^2(\Gamma_0)} \, dt \\
\displaystyle \le \| \widetilde \ell_0 \|_{L^2(0, \tau; L^2(\Gamma_0))}^2 + C_{{\cal F}^{\sigma}}^2 C_S^2 \int_{\tau}^{\tau'} t \left( \int_0^t \bigl\| {\cal R}(\sigma_{k}^3(\cdot, s)) (x') \bigr\|^2_{L^{\infty}(\Gamma_0)} \, ds \right) \, dt \\
\displaystyle \le \| \widetilde \ell_0 \|_{L^2(0, \tau; L^2(\Gamma_0))}^2 + C_{{\cal F}^{\sigma}}^2 C_S^2  (\tau' - \tau) T^2 C_{{\cal R}}^2 \| \sigma_k^3 \|^2_{L^{\infty} (0, \tau'; E ( \Omega))}.
\end{array}
\end{eqnarray*}
Similarly
\begin{eqnarray*}
\begin{array}{ll}
\displaystyle \left\| \frac{\partial \widetilde \ell_{k+1}}{\partial t}  \right\|^2_{L^2(0, \tau'; L^2(\Gamma_0) ) } 
= \left\| \frac{\partial  \widetilde \ell_0}{\partial t}  \right\|_{L^2(0, \tau; L^2(\Gamma_0))}^2 +  \left\| \frac{\partial  \widetilde \ell_{k+1}}{\partial t}  \right\|_{L^2(\tau, \tau'; L^2(\Gamma_0))}^2 \\
\displaystyle \le  \left\| \frac{\partial  \widetilde \ell_0}{\partial t}  \right\|_{L^2(0, \tau; L^2(\Gamma_0))}^2 
+ 3 \int_{\tau}^{\tau'} \left\| \frac{\partial {\cal F}^{\sigma}}{\partial t} (x',t)   \int_{0}^{t} S(t-s) \bigl|  {\cal R}(\sigma_{k}^3(\cdot, s)) (x') \bigr| \, ds \right\|^2_{L^2(\Gamma_0)} \, dt  \\
\displaystyle  + 3 \int_{\tau}^{\tau'}  \left\|  {\cal F}^{\sigma} (x',t)   S(0) \bigl| {\cal R}( \sigma_{k}^3(\cdot, t)) (x') \bigr| \right\|^2_{L^2(\Gamma_0)} \, dt \\
\displaystyle + 3 \int_{\tau}^{\tau'} \left\|  {\cal F}^{\sigma} (x',t)   \int_{0}^{t} S' (t-s) \bigl|  {\cal R}(\sigma_{k}^3(\cdot, s)) (x') \bigr| \, ds \right\|^2_{L^2(\Gamma_0)} \, dt \\
\displaystyle 
 \end{array}
 \end{eqnarray*}
 and thus 
\begin{eqnarray*}
\begin{array}{ll}
\displaystyle \left\| \frac{\partial \widetilde \ell_{k+1}}{\partial t}  \right\|^2_{L^2(0, \tau'; L^2(\Gamma_0) ) } \\
\displaystyle 
\le  \left\| \frac{\partial  \widetilde \ell_0}{\partial t}  \right\|_{L^2(0, \tau; L^2(\Gamma_0))}^2 
+ 3 C_S^2  T^2 C_{{\cal R}}^2 \| \sigma_k^3 \|^2_{L^{\infty} (0, \tau'; E ( \Omega))} \int_{\tau}^{\tau'}  \left\| \frac{\partial {\cal F}^{\sigma}}{\partial t} (x',t) \right\|^2_{L^2(\Gamma_0)} \, dt \\
\displaystyle
+ 
3 C_{{\cal F}^{\sigma}}^2 C_S^2  (\tau' - \tau) (  T^2 +1) C_{{\cal R}}^2 \| \sigma_k^3 \|^2_{L^{\infty} (0, \tau'; E ( \Omega))} \\
\displaystyle
\le  \left\| \frac{\partial  \widetilde \ell_0}{\partial t}  \right\|_{L^2(0, \tau; L^2(\Gamma_0))}^2 
+ 
3 C_{{\cal F}^{\sigma}}^2 C_S^2 (T^2 +1)   C_{{\cal R}}^2 \bigl(  (\tau' - \tau) + (\tau' - \tau)^{\frac{p-2}{p}} \bigr) \| \sigma_k^3 \|^2_{L^{\infty} (0, \tau'; E ( \Omega))} .
%
 \end{array}
 \end{eqnarray*}
 By combining these estimates we get 
\begin{eqnarray*}
\begin{array}{ll}
\displaystyle \bigl\| \widetilde \ell_{k+1} \bigr\|^2_{W^{1,2}(0, \tau'; L^2(\Gamma_0) ) }  
 \le \| \widetilde \ell_0 \|_{W^{1,2}(0, \tau; L^2(\Gamma_0))}^2  \\
 \displaystyle + C_{{\cal F}^{\sigma}}^2 C_S^2     C_{{\cal R}}^2 
2 (4 T^2 + 3)  C_{\rm data}^2 \bigl(  (\tau' - \tau) + (\tau' - \tau)^{\frac{p-2}{p}} \bigr) \\
\displaystyle
\times \bigl( 1 + \| {\cal F}^0 \|_{W^{1,2}(0,\tau'; L^2(\Gamma_0))} +  \| \widetilde \ell_k \|_{W^{1,2} (0, \tau'; L^2 (\Gamma_0))} \bigr)^2 
%
 \end{array}
 \end{eqnarray*}
 and we may conclude  with
\begin{eqnarray*}
\begin{array}{ll}
\displaystyle C_{\tau} =  \| \widetilde \ell_0 \|_{W^{1,2}(0, \tau; L^2(\Gamma_0))}  \\
\displaystyle + C_{{\cal F}^{\sigma}} C_S C_{{\cal R}}  \sqrt{ 2  (4 T^2 + 3)}  C_{\rm data} \bigl( \sqrt{T} + T^{\frac{p-2}{2p}} \bigr)  
  \bigl( 1 +\| {\cal F}^0 \|_{W^{1,2}(0,T; L^2(\Gamma_0))} 
  \bigr) 
\end{array}
 \end{eqnarray*}
 and
 \begin{eqnarray*}
\displaystyle C'_{\rm data}  =  C_{{\cal F}^{\sigma}} C_S  C_{{\cal R}}  \sqrt{2 (4 T^2 + 3) }
 C_{\rm data} .
 \end{eqnarray*}

\end{proof}

\bigskip

Let us fix now $\tau' > \tau$ such that 
\begin{eqnarray*} \label{tau}
 \sqrt{\tau' - \tau} + (\tau' - \tau)^{\frac{p-2}{2p}} 
\le  \frac{1}{ 2 C'_{\rm data}}.
\end{eqnarray*}
For instance we may choose
\begin{eqnarray*}
\begin{array}{ll}
\displaystyle \tau' = \tau + \frac{1}{ ( 4 C'_{\rm data})^2} \quad \hbox{ if $\displaystyle \frac{1}{ 4 C'_{\rm data}} \ge 1$,} \\
\displaystyle \tau' = \tau + \frac{1}{ ( 4 C'_{\rm data})^{\frac{2p}{p-2}}} \quad \hbox{ otherwise.}
\end{array}
\end{eqnarray*}
By observing that 
\begin{eqnarray*}
\displaystyle \| \widetilde \ell_0 \|_{W^{1,2}(0, \tau'; L^2(\Gamma_0))}  
\le C'_{\tau} \eqldef C_{\tau} +   C_{{\cal F}^{\sigma}} C_S   C_{{\cal R}} \tau \sqrt{T- \tau} \| \bar \sigma^3 \|_{L^{\infty}( 0,\tau; E(\Omega) )}
\end{eqnarray*} 
we obtain with an immediate induction that 
\begin{eqnarray*}
\bigl\| \widetilde \ell_{k} \bigr\|_{W^{1,2}(0, \tau'; L^2(\Gamma_0) ) } \le C'_{\tau} \sum_{m=0}^k \frac{1}{2^m} \le 2 C'_{\tau} \quad \forall k \ge 0.
\end{eqnarray*}
It follows that $\widetilde v_k $, $\displaystyle \frac{\partial^2 \widetilde v_k}{\partial t^2}$, $p_k$, $\sigma_k$ and $div(\sigma_k)$ are uniformly bounded   in $W^{1,\infty} \bigl(0, \tau'; {\bf L}^2(\Omega) \bigr) \cap W^{1,2} (0, \tau'; {\cal  V}_{0 div}) $, $L^2 \bigl(0, \tau', \bigl({\bf H}^1_{0 div} (\Omega) \bigr)' \bigr)$, $L^{\infty} \bigl(0, \tau'; L^2_0(\Omega) \bigr)$, $L^{\infty} \bigl(0, \tau';  (L^2(\Omega)^{3\times 3} \bigr)$ and in $L^{\infty} \bigl(0, \tau'; {\bf L}^2(\Omega) \bigr)$ respectively.

Hence, possibly modifying $\widetilde v_k $ and  $\displaystyle \frac{\partial \widetilde v_k}{\partial t}$ on a negligible subset of $[0, \tau']$,  we have $\widetilde v_k \in C^0 \bigl( [0, \tau']; {\cal V}_{0 div} \bigr)$ and $\displaystyle \frac{\partial \widetilde v_k}{\partial t} \in C^0 \bigl( [0, \tau']; \bigl({\bf H}^1_{0 div} (\Omega) \bigr)' \bigr)$. Moreover
\begin{eqnarray*}
\displaystyle
\widetilde{v}_{k } , \frac{\partial \widetilde{v}_{k }}{\partial t} \rightharpoonup \widetilde{v_*}, \frac{\partial \widetilde{v_*}}{\partial t}  \quad && \mbox{weakly star in $L^{\infty}\bigl(0, \tau' ;{\bf L}^{2}(\Omega) \bigr)$} \\
&& \mbox{and weakly in $ L^{2}(0,\tau' ;{\cal V}_{0 div})$,}
\end{eqnarray*}
\begin{eqnarray*}
\displaystyle{
\frac{\partial^2 \widetilde{v}_{k }}{\partial t^2} \rightharpoonup \frac{\partial^2 \widetilde{v_*}}{\partial t^2} \quad\mbox{weakly in $ L^{2} \bigl(0,\tau';\bigl({\bf H}^1_{0 div} (\Omega) \bigr)' \bigr) $, }}
\end{eqnarray*}
\begin{eqnarray*}
\displaystyle{
p_{k } \rightharpoonup p_*  \quad\mbox{weakly star in $L^{\infty}\bigl(0,\tau';{ L}^{2}_0(\Omega) \bigr)$,}}
\end{eqnarray*}
and 
\begin{eqnarray*}
\displaystyle{
\widetilde{\ell}_{k } \rightharpoonup \widetilde{\ell}_*  \quad\mbox{weakly  in $L^{2}\bigl(0,\tau';{ L}^{2}(\Gamma_0) \bigr)$.}}
\end{eqnarray*}
By using Aubin's and Simon's lemmas, and possibly extracting another subsequence, we have also 
\begin{eqnarray*} 
\widetilde{v}_{k} \rightarrow \widetilde{v}_* \quad\mbox{strongly in $ L^{2} \bigl(0,\tau';{\bf L}^{2}(\Gamma_0) \bigr)$,}
\end{eqnarray*}
\begin{eqnarray} \label{point-fixe5}
 \frac{\partial \widetilde{v}_{k}}{\partial t} \rightarrow \frac{\partial \widetilde{v}_*}{\partial t} \quad\mbox{strongly in $ L^{2} \bigl(0, \tau';{\bf L}^{2}(\Omega) \bigr)$,}
\end{eqnarray}
and
\begin{eqnarray} \label{point-fixe6}
\widetilde{v}_{k } \rightarrow \widetilde{v}_*  \quad\mbox{strongly in ${\cal C}^{0}\bigl( [0,\tau']; {\bf L}^2(\Omega) \bigr)$.}
\end{eqnarray}
Furthermore,  possibly modifying $\widetilde v_* $ and  $\displaystyle \frac{\partial \widetilde v_*}{\partial t}$ on a negligible subset of $[0, \tau']$,  we have $\widetilde v_* \in C^0 \bigl( [0, \tau']; {\cal V}_{0 div} \bigr)$ and $\displaystyle \frac{\partial \widetilde v_*}{\partial t} \in C^0 \bigl( [0, \tau']; \bigl({\bf H}^1_{0 div} (\Omega) \bigr)' \bigr)$.
By passing to the limit as $k$ tends to $+ \infty$ in (\ref{NS-25-k})-(\ref{NS-25init-k}) we get 
\begin{eqnarray*} 
\begin{array}{ll}
\displaystyle
\left\langle \frac{d}{dt} \left( \widetilde{v}_*, \varphi \right) , \chi\right\rangle_{{\cal D}'(0,\tau'), {\cal D}(0,\tau')}  
 - \bigl\langle \bigl(p_*,div(\varphi)  \bigr), \chi \bigr\rangle_{{\cal D}'(0,\tau'), {\cal D}(0,\tau')} \\
\displaystyle + \int_0^{\tau'} a(\widetilde{v}_*,\varphi \chi ) \, dt 
+  \int_0^{\tau'} \int_{\Gamma_0} \ell_* \bigl( | \widetilde{v}_*+ \varphi \chi \bigr| -  | \widetilde{v}_* | \bigr) \, dx' dt \\
\displaystyle  \geq
\bigl\langle (f,\varphi ), \chi \bigr\rangle_{{\cal D}'(0,\tau'), {\cal D}(0,\tau')} -  \int_0^{\tau'} a(v^{0}\zeta ,\varphi \chi) \, dt  
\displaystyle - \left\langle \left(v^{0}\frac{\partial \zeta}{\partial t},\varphi \right),
\chi \right\rangle_{{\cal D}'(0,\tau' ), {\cal D}(0,\tau')} 
\end{array}
\end{eqnarray*}
for all $\varphi \in {\cal V}_0$ and for all $\chi \in {\cal D}(0, \tau')$, with
\begin{eqnarray*} 
\widetilde{v}_*(0, \cdot) = \widetilde{v}^{0} =0.
\end{eqnarray*}
It follows that $(\widetilde v_*, p_*)$ is the unique solution on $[0, \tau']$ of problem $(P)$ with ${\cal F} = \ell_*= {\cal F}^0 + \widetilde{\ell}_* $ and $\sigma_* = - p_* {\rm Id} + 2 \mu D(\widetilde v_* + v^0 \zeta)$ satisfies $\sigma_* \in L^{\infty} \bigl( 0,\tau'; \bigl( L^2(\Omega) \bigr)^{3 \times 3} \bigr)$. Moreover, with a straighforward adaptation of Proposition \ref{proposition1}, we have also $div(\sigma_*) \in L^{\infty} \bigl( 0, \tau'; {\bf L}^{2} (\Omega) \bigr)$.
%

\bigskip

Let us prove now that 

\begin{lemma} \label{ell*}
Under the previous assumptions
\begin{eqnarray*}
\ell_{*} (x',t)  =  {\cal F}(x', t, \sigma_{*}) ={\cal F}^0(x',t) +  {\cal F}^{\sigma} (x',t)   \int_{0}^{t} S(t-s) \bigl| {\cal R}(\sigma_{*}^3(\cdot, s)) (x')\bigr| \, ds 
\end{eqnarray*}
for almost every $x' \in \Gamma_0$ and for all $t \in [0,\tau']$.
\end{lemma}

\begin{proof}
Let  $\varphi \in \bigl({\cal D}(\Omega)\bigr)^3$ and $\chi \in {\cal D}(0, \tau')$. We have
\begin{eqnarray} \label{point-fixe3}
\begin{array}{ll}
\displaystyle \left\langle {d \over d t} \left( \widetilde{v}_*, \varphi \right) ,\pm \chi\right\rangle_{{\cal D}'(0,\tau'), {\cal D}(0,\tau')} 
+ \int_0^{\tau'} a(\widetilde{v}_*+ v^0 \zeta, \pm \varphi \chi) \, dx dt \\
\displaystyle 
- \bigl\langle \bigl(p_* , div(\varphi)  \bigr), \pm \chi \bigr\rangle_{{\cal D}'(0,\tau'), {\cal D}(0,\tau')} 
 \\
\displaystyle  \geq
\bigl\langle (f,\varphi ), \pm \chi \bigr\rangle_{{\cal D}'(0,\tau'), {\cal D}(0,\tau')} 
 - \left\langle \left(v^{0} \frac{\partial \zeta}{\partial t} ,\varphi \right),
 \pm \chi \right\rangle_{{\cal D}'(0,\tau'), {\cal D}(0,\tau')}  .
\end{array}
\end{eqnarray}
Similarly, for all $k \ge 0$ we have also
\begin{eqnarray} \label{point-fixe4}
\begin{array}{ll}
\displaystyle \left\langle {d \over d t} \left( \widetilde{v}_k, \varphi \right) ,\pm \chi\right\rangle_{{\cal D}'(0,\tau'), {\cal D}(0,\tau')} 
+ \int_0^{\tau'} a(\widetilde{v}_k+ v^0 \zeta,\pm \varphi \chi) \, dx dt  \\
\displaystyle 
- \bigl\langle \bigl(p_k , div(\varphi)  \bigr), \pm \chi \bigr\rangle_{{\cal D}'(0,\tau'), {\cal D}(0,\tau')}  \\
\displaystyle   \geq
\bigl\langle (f,\varphi ), \pm \chi \bigr\rangle_{{\cal D}'(0,\tau'), {\cal D}(0,\tau')} 
 - \left\langle \left(v^{0} \frac{\partial \zeta}{\partial t},\varphi \right),
 \pm \chi \right\rangle_{{\cal D}'(0,\tau'), {\cal D}(0,\tau')}  .
\end{array}
\end{eqnarray}
By substracting (\ref{point-fixe3}) to (\ref{point-fixe4}) we obtain
\begin{eqnarray*} 
\displaystyle \int_0^{\tau'} \int_{\Omega}  {\partial ( \widetilde{v}_k - \widetilde{v}_* ) \over\partial t} \cdot \varphi   \chi \, dx dt
+ \int_0^{\tau' } \int_{\Omega} \sum_{i, j =1}^3 ( \sigma_{ij k} - \sigma_{ij *} ) \frac{\partial \varphi_i}{\partial x_j} \chi \, dx dt
=0 \quad \forall k \ge 0.
\end{eqnarray*}
Thus
\begin{eqnarray*}
 \bigl\| div( \sigma_k - \sigma_* ) \bigr\|_{L^2(0, \tau'; {\bf L}^2(\Omega) )}  
\displaystyle = \left\|  {\partial (\widetilde{v_k} - \widetilde{v_*} ) \over\partial t} \right\|_{L^2(0, \tau'; {\bf L}^2(\Omega) )} \quad \forall k \ge 0.
\end{eqnarray*}
Since $\displaystyle \left( {\partial \widetilde{v}_k  \over\partial t} \right)_{k \ge 0}$ converges strongly in $L^2 \bigl( 0, \tau'; {\bf L}^2(\Omega) \bigr)$ to $\displaystyle {\partial \widetilde{v}_*  \over\partial t}$, we infer that $\bigl( div(\sigma_k) \bigr)_{k \ge 0}$ converges also strongly to $div(\sigma_*)$ in $L^2 \bigl( 0, \tau'; {\bf L}^2(\Omega) \bigr)$. Hence
\begin{eqnarray*}
\begin{array}{ll}
\displaystyle \left\| {\partial \widetilde{v}_k  \over\partial t} - {\partial \widetilde{v}_*  \over\partial t} \right\|_{{\bf L}^2 (\Omega) } \to 0 \quad \hbox{ strongly in $L^2(0, \tau')$,} \\
\displaystyle \bigl\| div(\sigma_k) - div (\sigma_*) \bigr\|_{{\bf L}^2 (\Omega) } \to 0 \quad \hbox{ strongly in $L^2(0, \tau')$,}
\end{array}
\end{eqnarray*}
and we infer that, possibly extracting a subsequence still denoted $({\widetilde v}_k, p_k)_{k \ge 0}$, there exists a negligible subset $A$ of $(0, \tau')$ such that 
\begin{eqnarray} \label{final1}
\begin{array}{ll}
\displaystyle  {\partial \widetilde{v}_k  \over\partial t} (t) , div(\sigma_k) (t) \to  {\partial \widetilde{v}_*  \over\partial t} (t) , div (\sigma_*) (t)  \\
\displaystyle  \quad \qquad  \hbox{ strongly in ${\bf L}^2(\Omega)$, for all $t \in (0 , \tau') \setminus A$.} 
\end{array}
\end{eqnarray}
On the other hand, $({\widetilde v}_k )_{k \ge 0}$ is bounded in $W^{1,2} (0, \tau'; {\cal V}_{0 div})$ and by using Helly's theorem (see \cite{mmm} for instance) we obtain that, possibly extracting another subsequence, still denoted $({\widetilde v}_k, p_k)_{k \ge 0}$, we have
\begin{eqnarray} \label{final2}
{\widetilde v}_k (t) \rightharpoonup \Lambda (t) \quad \hbox{ weakly in ${\cal V}_{0 div}$, for all $t \in [0, \tau']$}
\end{eqnarray}
with $\Lambda \in BV(0, \tau'; {\cal V}_{0 div})$. Then, for all $\varphi \in {\cal V}_{0 div}$ and for all $\chi \in {\cal D} (0, \tau')$ we have
\begin{eqnarray*}
  \bigl( {\widetilde v}_k (t) , \varphi \bigr)_{{\bf H}^1(\Omega)}   \chi(t)  \to \bigl(\Lambda (t), \varphi \bigr)_{{\bf H}^1(\Omega)} \chi (t) 
\quad \hbox{ for all $t \in [0, \tau']$}
\end{eqnarray*}
where $( \cdot, \cdot)_{{\bf H}^1(\Omega)}$ denotes the inner product of ${\bf H}^1(\Omega)$ and 
\begin{eqnarray*}
\bigl|   \bigl( {\widetilde v}_k (t) , \varphi \bigr)_{{\bf H}^1(\Omega)}   \chi(t)  \bigr| 
\le  \| \chi\|_{C^0([0, \tau'])} \| \varphi \|_{{\bf H}^1 (\Omega)} \|  {\widetilde v}_k \|_{L^{\infty} (0, \tau'; {\bf H}^1 (\Omega))}
\quad \hbox{ for all $t \in [0, \tau']$.}
\end{eqnarray*}
We may apply Lebesgue's dominated theorem and we get
\begin{eqnarray*}
\int_0^{\tau'}  \bigl( {\widetilde v}_k (t) , \varphi \bigr)_{{\bf H}^1(\Omega)}   \chi(t)  \, dt
\to \int_0^{\tau'}   \bigl(\Lambda (t), \varphi \bigr)_{{\bf H}^1(\Omega)} \chi (t) \, dt .
\end{eqnarray*}
Owing that $({\widetilde v}_k )_{k \ge 0}$ converges weakly to ${\widetilde v}_*$  in $L^{2} (0, \tau'; {\cal V}_{0 div})$, we infer that $\Lambda = {\widetilde v}_*$ in $L^{2} (0, \tau'; {\cal V}_{0 div})$, i.e.
\begin{eqnarray*}
\int_0^{\tau'} \bigl\| \Lambda (t) - {\widetilde v}_* (t) \bigr\|^2_{{\bf H}^1(\Omega)} \, dt = 0 . 
\end{eqnarray*}
It follows that there exists another negligible subset $A'$ of $(0, \tau')$ such that
\begin{eqnarray} \label{final2bis}
\Lambda (t) = {\widetilde v}_* (t) \quad  \hbox{ in ${\cal V}_{0 div}$, for all $t \in (0, \tau') \setminus A'$.}
\end{eqnarray}
Recalling that,   $\displaystyle \frac{\partial \widetilde v_k}{\partial t}$ belongs to $C^0 \bigl( [0,\tau'], \bigl({\bf H}^1_{0 div} (\Omega) \bigr)' \bigr) \cap L^{\infty} \bigl(0, \tau', {\bf L}^2 (\Omega) \bigr)$ for all $k \ge 0$, we infer that $\displaystyle \frac{\partial \widetilde v_k}{\partial t}$ is weakly continuous with values in ${\bf L}^2(\Omega)$ on $[0, \tau']$ and 
\begin{eqnarray*}
\left\| \frac{\partial \widetilde v_k}{\partial t} ( t) \right\|_{{\bf L}^2(\Omega)} \le \left\| \frac{\partial \widetilde v_k}{\partial t}  \right\|_{L^{\infty} (0, \tau'; {\bf L}^2(\Omega)) } \quad \hbox{ for all $t \in [0, \tau']$}
\end{eqnarray*}
(see lemma 1.4 page 263 in \cite{temam}). Now, let $k \ge 0$. For all $t \in [0,\tau']$ we define $f_k(t) \in {\bf H}^{-1}(\Omega)$  by
\begin{eqnarray*}
\begin{array}{ll}
\displaystyle \bigl\langle f_k (t) , \varphi \bigr\rangle_{{\bf H}^{-1} (\Omega), {\bf H}^1_0 (\Omega) }  = \left( \frac{\partial \widetilde v_k}{\partial t} ( t) + v^0 \frac{\partial \zeta}{\partial t} (t) , \varphi \right)  + a \bigl( \widetilde v_k (t) + v^0 \zeta (t) , \varphi \bigr) \\
\displaystyle  - \bigl(f (t), \varphi \bigr) \quad \forall \varphi \in {\bf H}^1_0 (\Omega).
\end{array}
\end{eqnarray*}
Then, consider now $\varphi \in {\bf H}^1_{0 div} ( \Omega) $. With (\ref{NS-25-k}), we obtain that
\begin{eqnarray*}
\int_0^{\tau'} \bigl\langle f_k (t) , \varphi \bigr\rangle_{{\bf H}^{-1} (\Omega), {\bf H}^1_0 (\Omega) } \chi (t) \, dt = 0
\quad \forall \chi \in {\cal D} (0, \tau').
\end{eqnarray*}
So 
\begin{eqnarray*}
 \bigl\langle f_k (t) , \varphi \bigr\rangle_{{\bf H}^{-1} (\Omega), {\bf H}^1_0 (\Omega) } = 0 \quad \hbox{a.e. in $(0, \tau')$}
 \end{eqnarray*}
and, using the continuity of the mapping $t \mapsto  \bigl\langle f_k (t) , \varphi \bigr\rangle_{{\bf H}^{-1} (\Omega), {\bf H}^1_0 (\Omega) }$ on $[0, \tau']$, we infer that the previous equality is valid for all $t \in [0, \tau']$. It follows that there exists a mapping $\widetilde p_k : [0, \tau'] \to L^2_0 (\Omega)$ such that, for all $t \in [0, \tau']$ 
\begin{eqnarray*}
 \bigl\langle f_k (t) , \varphi \bigr\rangle_{{\bf H}^{-1} (\Omega), {\bf H}^1_0 (\Omega) } = \bigl\langle \nabla \widetilde p_k (t) , \varphi \bigr\rangle_{{\cal D}'(\Omega), {\cal D}(\Omega)}
 \quad \forall \varphi \in {\cal D}(\Omega).
 \end{eqnarray*}
 But, for all $t \in [0, \tau']$, we have $\widetilde p_k (t) \in L^2_0(\Omega)$ and thus
 \begin{eqnarray*}
 \begin{array}{ll}
\displaystyle  \bigl\langle \nabla \widetilde p_k (t) , \varphi \bigr\rangle_{{\cal D}'(\Omega), {\cal D}(\Omega)} = - \bigl\langle \widetilde p_k (t) , div (\varphi) \bigr\rangle_{{\cal D}'(\Omega), {\cal D}(\Omega)} \\
 \displaystyle = - \bigl( \widetilde p_k (t) , div (\varphi) \bigr) \quad  \forall \varphi \in {\cal D}(\Omega).
\end{array}
 \end{eqnarray*}
It follows that, for all $t \in [0,\tau']$, 
\begin{eqnarray*}
\begin{array}{ll}
\displaystyle - \bigl( \widetilde p_k (t) , div (\varphi) \bigr) = \left( \frac{\partial \widetilde v_k}{\partial t} ( t) + v^0 \frac{\partial \zeta}{\partial t} (t) , \varphi \right)  + a \bigl( \widetilde v_k (t) + v^0 \zeta (t) , \varphi) \\
\displaystyle  - \bigl(f (t), \varphi \bigr) \quad  \forall \varphi \in {\cal D}(\Omega)
\end{array}
\end{eqnarray*}
and by density of ${\cal D}(\Omega)$ into ${\bf H}^1_0(\Omega)$, the same equality is valid for all $\varphi \in {\bf H}^1_0(\Omega)$.
With the same arguments as in Theorem \ref{tresca-existence} and Theorem \ref{tresca-regularity}, we obtain also that 
$\widetilde p_k \in L^{\infty} \bigl(0, \tau'; L^2_0(\Omega) \bigr)$
and $p_k = \widetilde p_k$ in $L^{\infty} \bigl(0, \tau'; L^2_0(\Omega) \bigr)$. Thus possibly modifying $p_k$ on a negligible subset of $(0, \tau')$ we have
\begin{eqnarray} \label{final3bis}
\begin{array}{ll}
\displaystyle - \bigl(  p_k (t) , div (\varphi) \bigr) = \left( \frac{\partial \widetilde v_k}{\partial t} ( t) + v^0 \frac{\partial \zeta}{\partial t} (t), \varphi \right)  + a \bigl( \widetilde v_k (t) + v^0 \zeta (t) , \varphi \bigr) \\
\displaystyle  - \bigl(f (t), \varphi \bigr) \quad  \forall \varphi \in {\bf H}^1_0 (\Omega), \quad \forall t \in [0, \tau'].
\end{array}
\end{eqnarray}
Similarly, possibly modifying $p_*$ on a negligible subset of $(0, \tau')$, we have 
\begin{eqnarray} \label{final3ter}
\begin{array}{ll}
\displaystyle - \bigl(  p_* (t) , div (\varphi) \bigr) = \left( \frac{\partial \widetilde v_*}{\partial t} ( t) + v^0 \frac{\partial \zeta}{\partial t} (t), \varphi \right)  + a \bigl( \widetilde v_* (t) + v^0 \zeta (t) , \varphi \bigr) \\
\displaystyle  - \bigl(f (t), \varphi \bigr) \quad  \forall \varphi \in {\bf H}^1_0 (\Omega), \quad \forall t \in [0, \tau'].
\end{array}
\end{eqnarray}

\bigskip

Now let $\tilde w \in L^2(\Omega)$ and $w \in L^2_0 (\Omega)$ be given by
\begin{eqnarray*}
w= \tilde w - \frac{1}{|\Omega|} \int_{\Omega} \tilde w \, dx  .
\end{eqnarray*}
For all $k \ge 0$ and for all $t \in [0, \tau']$ we have
\begin{eqnarray*}
\begin{array}{ll}
\displaystyle  \bigl(  p_k (t) - p_* (t)  , \tilde w \bigr) =  \bigl(  p_k (t) - p_* (t)  ,  w \bigr)
 \\
 \displaystyle
 = 
 -  \left( \frac{\partial \widetilde v_k}{\partial t} (\cdot, t) -\frac{\partial \widetilde v_*}{\partial t} (\cdot, t), P(w) \right)   - a \bigl( \widetilde v_k (t) -\widetilde v_* (t)  , P(w) \bigr) 
%
 \end{array}
 \end{eqnarray*}
 where $P$ is the linear continuous operator from $ L^2_0(\Omega)$ to $ {\bf H}^1_0(\Omega)$ such that $div \bigl( P(w) \bigr) = w $ for all $w \in L^2_0(\Omega)$ (\cite{girault}).
With (\ref{final1}) and (\ref{final2})-(\ref{final2bis}) we get
\begin{eqnarray*}
 \int_{\Omega} \bigl( p_k (t) - p_* (t) \bigr) \tilde w  \, dx \to 0 \quad \hbox{ for all $\tilde w \in L^2(\Omega)$, for all $t \in (0, \tau') \setminus (A \cup A')$}
 \end{eqnarray*}
 which implies that
 \begin{eqnarray} \label{final3bis}
p_k(t) \rightharpoonup p_* (t) \quad \hbox{ weakly in $L^2(\Omega)$, for all $t \in (0, \tau') \setminus (A \cup A') $.}
\end{eqnarray}
 Then with (\ref{final1}), (\ref{final2})-(\ref{final2bis}) and (\ref{final3bis}) we may conclude that
\begin{eqnarray*}
\sigma_k (t) \rightharpoonup \sigma_* (t) \quad \hbox{ weakly in ${\bf E}(\Omega)$, for all $t \in (0, \tau') \setminus (A \cup A')$.}
\end{eqnarray*}


 \bigskip

 By using the definition of ${\cal R}$, we obtain that
 \begin{eqnarray*}
 \begin{array}{ll}
\displaystyle  {\cal R} \bigl( \sigma_k^3 (\cdot, t)  \bigr) (x')  = \bigl\langle \gamma_n \bigl(  \sigma_k^3 (t) \bigr) , f_{x'} \bigr\rangle_{H^{-1/2} (\partial \Omega), H^{1/2} (\partial \Omega)} \\
 \displaystyle \to {\cal R} \bigl( \sigma_*^3 (\cdot, t) \bigr) (x') \quad \hbox{ for all $x' \in \Gamma_0$, for all $t \in (0, \tau') \setminus (A \cup A')$.}
 \end{array}
 \end{eqnarray*}
 Moreover, for all $k \ge 0$, 
 \begin{eqnarray*}
\bigl|  {\cal R} \bigl( \sigma_k^3 (\cdot, t) \bigr) (x')  \bigr| \le C_{{\cal R}} \bigl\| \sigma_k^3 ( \cdot, t) \bigr\|_{E(\Omega)} \le C_{{\cal R}} \| \sigma_k^3  \|_{L^{\infty} (0, \tau'; E(\Omega))}
\end{eqnarray*}
 for all $x' \in \Gamma_0$, for almost every $t \in [0, \tau']$,  i.e. there exists a negligible subset $A_k$ of $(0, \tau')$ such that 
  \begin{eqnarray*}
\bigl|  {\cal R} \bigl( \sigma_k^3 (\cdot, t) \bigr) (x')  \bigr| \le C_{{\cal R}} \bigl\| \sigma_k^3 ( \cdot, t) \bigr\|_{E(\Omega)} \le C_{{\cal R}} \| \sigma_k^3  \|_{L^{\infty} (0, \tau'; E(\Omega))}
\end{eqnarray*}
 for all $x' \in \Gamma_0$, for all $t \in (0, \tau') \setminus A_k$. Let $\displaystyle A'' = A \cup A' \cup \bigl( \cup_{k \ge 0} A_k \bigr)$.
 By applying Lebesgue's dominated theorem, we get
 \begin{eqnarray*}
\bigl| {\cal R} ( \sigma_k^3 ) (\cdot,t) \bigr| \to \bigl| {\cal R} ( \sigma_*^3) (\cdot,t) \bigr| \quad \hbox{ strongly in  $L^2( \Gamma_0)$, for all $t \in (0, \tau') \setminus A''$}
 \end{eqnarray*}

 Thus, observing that $A''$ is also a negligible subset of $(0, \tau')$, we may apply twice Lebesgue's dominated theorem and we get
 \begin{eqnarray*}
 \begin{array}{ll}
\displaystyle \int_0^t S(t-s) \bigl| {\cal R} ( \sigma_k^3 ) (\cdot,s) \bigr| \, ds  \to \int_0^t S (t-s) \bigl| {\cal R} ( \sigma_*^3) (\cdot,s) \bigr| \, ds \\
\displaystyle \qquad \qquad \hbox{ strongly in  $L^2( \Gamma_0)$, for all $t \in [0, \tau']$}
\end{array}
 \end{eqnarray*}
and
\begin{eqnarray*}
\begin{array}{ll}
\displaystyle \int_0^{\star} S(\star -s) \bigl| {\cal R} ( \sigma_k^3 ) (\cdot,s) \bigr| \, ds 
\to \int_0^{\star} S(\star -s) \bigl| {\cal R} ( \sigma_*^3 ) (\cdot,s) \bigr| \\
\displaystyle \qquad  \quad \hbox{ strongly in $L^2\bigl( 0, \tau', L^2(\Gamma_0) \bigr)$}
\end{array}
\end{eqnarray*}
which allows us to conclude.

\end{proof}

Gathering all the previous results we have  obtained that $({\widetilde v}_*, p_*)$ is a solution of problem $(P)$ on $[0, \tau']$ and $({\widetilde v}_*, p_*)$ is an extension of $(\bar v, \bar p)$ to $[\tau, \tau']$. Indeed, $\ell_k = \ell_0 ={\cal F} (\cdot, \cdot, \bar \sigma)$ on $(0, \tau)$ for all $k \ge 0$, thus $\ell_* = {\cal F} (\cdot, \cdot, \bar \sigma)$ on $(0, \tau)$. Moreover $\tau' - \tau$ is independent of $\tau$. Thus, starting from $\tau =0$, we may conclude with a  finite induction argument that

\begin{theorem} \label {coulomb-existence}
Let assumptions (\ref{zeta})-(\ref{v0-1})-(\ref{v0})-(\ref{ell})-(\ref{coulomb.2})-(\ref{coulomb.2bis}) hold. Then, the non-local friction problem $(P)$ admits a solution $({\widetilde v}, p)$ such that 
\begin{eqnarray*}
\widetilde v \in W^{1,\infty} \bigl( 0,T; {\bf L}^2 (\Omega) \bigr) \cap W^{1,2} (0,T; {\cal V}_{0 div}), \quad p \in L^{\infty} \bigl(0,T; L^2_0 (\Omega) \bigr)
\end{eqnarray*}
and
\begin{eqnarray*}
 \frac{\partial^2 \widetilde v}{\partial t^2} \in  L^2 \bigl(0,T; \bigl( {\bf H}^1_{0 div}(\Omega) \bigr)' \bigr).
\end{eqnarray*}
\end{theorem}



\bigskip
\bigskip

\bibliographystyle{elsarticle-num}

\begin{thebibliography}{00}

\bibitem{barrat}
{\sc J.~Barrat,  L.~Bocquet, }
\emph{Large slip effect at a nonwetting fluid-solid interface.}
 Phys. Rev. Lett., 820(23) (1999) 4671--4674.
 
 \bibitem{baudry}
{\sc J.~Baudry,  E.~Charlaix,  A.~Tonck,  D.~Mazuyer,}
\emph{Experimental evidence for a large slip effect at a nonwetting fluid-solid
interface.}
 Langmuir, 17(17) (2001) 5232--5236. 

\bibitem{bonaccurso}
{\sc E.~Bonaccurso,  H.~Butt,  V.~Craig,}
\emph{Surface roughness and hydrodynamic boundary slip of a Newtonian fluid in a
completely wetting system.}
 Phys. Rev. Lett., 90(14) (2003) 144501.

  
\bibitem{brezina}
{\sc J.~Brezina,}
\emph{Asymptotic properties of solutions to the equations of incompressible fluid mechanics.}
 J. Math. Fluid Mech., 12 (2010) 536--553.

 
  
\bibitem{bucur1}
{\sc  D.~Bucur,  E.~Feireisl,  S.~Necasova,}
\emph{Influence of wall roughness on the slip behaviour of viscous fluids.}
 Proc. Roy. Soc.
Edinburgh Sect. A, 138 (2008) 957--973. 



\bibitem{bucur2}
{\sc  D.~Bucur,  E.~Feireisl,  S.~Necasova,}
\emph{On the asymptotic limit of flows past a ribbed boundary.}
 J. Math. Fluid Mech., 10(4) (2008) 554--568.
 

\bibitem{bucur3}
{\sc  D.~Bucur,  E.~Feireisl,  S.~Necasova, J.~Wolf,}
\emph{On the asymptotic limit of the Navier-Stokes system on domains with
rough boundaries.}
 J. Differ. Equ., 244(11) (2008) 2890--2908.


  
\bibitem{coddington}
{\sc E.A.~Coddington, N.~Levinson,}
Theory of ordinary differential equations, Mc Graw-Hill, New-York, 1955.

\bibitem{consiglieri}
{\sc L.~Consiglieri,}
\emph{Existence for a class of non-Newtonian fluids
with a nonlocal friction boundary condition.}
Acta Math. Sin. Eng. Series, 22(2) (2006) 523--534.


\bibitem{coulomb}
{\sc C.A.~Coulomb,}
  Th\'eorie des machines simples,
 Bachelier, Paris, 1821.


\bibitem{demkowicz}
{\sc L.~Demkowicz, J.T.~Oden,}
\emph{On some existence and uniqueness results in 
contact problems with nonlocal friction.}
Nonlinear Analysis TMA, 6(10) (1982) 1075--1093.
  
\bibitem{duvaut2} 
{\sc G.~Duvaut,} 
\emph{\'Equilibre d'un solide \'elastique avec contact unilat\'eral  et 
frottement de Coulomb.}
C. R. Acad. Sci. Paris, t 290 (1980) 263--265.

\bibitem{duvaut3} 
{\sc G.~Duvaut,} 
\emph{Loi de frottement non locale.}
Journal de M\'ecanique Th\'eorique et Appliqu\'ee, Num\'ero sp\'ecial (1982) 73--78.

\bibitem{fujita1}
{\sc H.~Fujita,}
\emph{Flow Problems with Unilateral Boundary Conditions.}
 Le\c cons, Coll\`ege de France (1993).

 
\bibitem{fujita2}
{\sc H.~Fujita,}
\emph{A mathematical analysis of motions of viscous incompressible fluid under leak or slip boundary conditions.}
 Math. Fluid Mech. Model., 888
(1994) 199--216.

 
\bibitem{fujita3}
{\sc H.~Fujita,}
\emph{Non-stationary Stokes flows under leak boundary conditions of friction type.}
 J. Comput. Math., 19 (2001) 1--8.


\bibitem{fujita4}
{\sc H.~Fujita,}
\emph{A coherent analysis of Stokes flows under boundary conditions of friction type.}
 J. Comput. Appl. Math., 149 (2002) 57--69.

  
\bibitem{girault}
{\sc V.~Girault and  P.A.~Raviart,}
 Finite Element Approximation of the Navier-Stokes Equations.
Springer-Verlag, Berlin, 1979.
  
\bibitem{haslinger}
{\sc J.~Haslinger, J.~Stebel and T.~Sassi,}
\emph{Shape optimization for Stokes problem with threshold slip.}
 Appl. Mathematics, 59(6) (2014) 631--652.


  
\bibitem{hervet}
{\sc H.~Hervet, L.~L\'eger,}
\emph{Flow with slip at the wall: from simple to complex fluids}, C. R. Acad. 
Sci. Paris, Physique 4 (2003) 241--249.

\bibitem{ionescu-sofonea}
{\sc I.R.~Ionescu, M.~Sofonea,}
\emph{The blocking property in the study of Bingham fluids},
Int. J. Engng. Sci., 24(3) (1986) 289--297.

  \bibitem{Korn}
{\sc V.A.~Kondrat'ev, O.A.~Oleinik,}
\emph{Boundary-value problems for the system of elasticity theory in
  unbounded domains. Korn's inequailities}.
 Russian Math. Surveys, 43(5) (1988) 65--119.
 
\bibitem{navier}
{\sc C.~Navier,}
\emph{M\'emoire sur les lois du mouvement des fluides.}
 M\'em. Acad. Royal Soc., 6 (1823) 389--416.

\bibitem{magnin}
{\sc A.~Magnin,  J.M.~Piau,}
\emph{Shear rheometry of fluids with a yield stress.}
 J. Non-Newtonian Fluid Mech., 23 (1987) 91--106.



 \bibitem{mmm}
{\sc  M.~Monteiro Marques,}
  Differential inclusions in non-smooth mechanical problems: shocks and dry friction,
     Birkhauser, Boston, Berlin, 1993.

\bibitem{richardson}
{\sc S.~Richardson,}
\emph{On the no-slip boundary condition}. 
J. Fluid Mech. 59 (1973) 707--719. 

\bibitem{rockafellar}
{\sc   R.T.~Rockafellar,}
 Convex analysis,
  Princeton University Press, Princeton,
1970.


\bibitem{roubicek}
{\sc T.~Roubicek,}
Nonlinear Partial Differential Equations with Applications.
Birkhauser-Verlag, Basel, Boston, Berlin, 2005.

\bibitem{saito1}
{\sc N.~Saito, H.~Fujita,}
\emph{Regularity of solutions to the Stokes equation under a certain nonlinear boundary condition.}
 The Navier-Stokes equations, Lecture
Note Pure Appl. Math., 223 (2001) 73--86.


\bibitem{saito2}  
{\sc  N.~Saito,}
\emph{On the Stokes equations with the leak and slip boundary conditions of friction type: regularity of solutions.}
 Publ. RIMS Kyoto Univ., 40 (2004)
345--383.


  
  \bibitem{sofonea}
{\sc M.~Sofonea, A.~Farcas,}
\emph{ Analysis of a History-Dependent Frictional Contact Problem.}
 Appl. Anal. 93 (2014), no.2, 428--444.
 
 \bibitem{temam}
{\sc R.~Temam,}
Navier-Stokes Equations.
North-Holland, Amsterdam, New-York, Oxford, 1984.


\bibitem{tretheway}
{\sc D.~Tretheway,  C.~Meinhart,}
\emph{Apparent fluid slip at hydrophobic microchannel walls.}
 Phys. Fluids, 14 (2002) L9.



\bibitem{zhu1}
{\sc Y.~Zhu,  S.~Granick,}
\emph{Limits of the hydrodynamic no-slip boundary condition.}
 Phys. Rev. Lett., 88(10) (2002) 106102.

\bibitem{zhu2}
{\sc Y.~Zhu,  S.~Granick,}
\emph{No-slip boundary condition switches to partial slip when fluid contains surfactant.}
 Langmuir,
18(26) (2002) 10058--10063. 

 
   \end{thebibliography}


\end{document}